\documentclass[11pt]{article}

\usepackage{amsthm}
\usepackage[T1]{fontenc}
\usepackage[utf8]{inputenc}
\usepackage[english]{babel}
\usepackage{graphicx, color}
\usepackage{amssymb}
\usepackage{amsmath, xfrac}
\usepackage{latexsym}
\usepackage{caption, subcaption}
\usepackage{lineno}
\usepackage{comment}
\usepackage[hyphens]{url}
\usepackage{hyperref}
\hypersetup{
    colorlinks = true,
    citecolor = blue,
    linkcolor = blue,
    urlcolor = black
}

\newtheorem{thm}{Theorem} 
\newtheorem{lem}[thm]{Lemma} 
\newtheorem{conj}[thm]{Conjecture}

\begin{document}

\title{Crossing and intersecting families of geometric graphs on point sets\footnote{This paper is a follow-up of the conference version of \cite{alvarezrebollar2015crossing}}}

\author{J.L. \'Alvarez-Rebollar\thanks{Posgrado en Ciencias Matemáticas, UNAM and Departamento de Ciencias Básicas, Instituto Tecnológico de Zitácuaro. Email: jose.ar@zitacuaro.tecnm.mx. Supported by SEP-CONACyT of Mexico.} \and
        J. Cravioto-Lagos \thanks{Posgrado en Ciencia e Ingenier\'ia de la Computaci\'on, UNAM. Email: jorge.cravioto@hotmail.com. Supported by SEP-CONACyT of Mexico.}\and
        N.~Mar\'in\thanks{Posgrado en Ciencia e Ingenier\'ia de la Computación, UNAM. Email: nestaly@ciencias.unam.mx. Supported by SEP-CONACyT of Mexico.} \and
        O.~Sol\'e-Pi\thanks{Facultad de Ciencias, UNAM. Email: oriolandreu@ciencias.unam.mx.} \and
        J. Urrutia\thanks{Instituto de Matem\'aticas, UNAM. Email: urrutia@matem.unam.mx. Partially supported by PAPIIT IN105221, Programa de Apoyo a la Investigación e Innovación Tecnológica UNAM.}
}

\maketitle
%\linenumbers
\begin{abstract}
Let $S$ be a set of $n$ points in the plane in general position. 
Two line segments connecting pairs of points of $S$ \emph{cross} if they have an interior point in common. 
Two vertex disjoint geometric graphs with vertices in $S$ \textit{cross} if there are two edges, one from each graph, which cross. 
A set of vertex disjoint geometric graphs with vertices in $S$ is called \emph{mutually crossing} if any two of them cross.

We show that there exists a constant $c$ such that from any family of $n$ mutually crossing triangles, one can always obtain a family of at least $n^c$ mutually crossing $2$-paths (each of which is the result of deleting an edge from one of the triangles) and then provide an example that implies that $c$ cannot be taken to be larger than $2/3$. 
For every $n$ we determine the maximum number of crossings that a Hamiltonian cycle on a set of $n$ points might have. 
Next, we construct a point set whose longest perfect matching contains no crossings. 
We also consider edges consisting of a horizontal and a vertical line segment joining pairs of points of $S$, which we call \emph{elbows}, and prove that in any point set $S$ there exists a family of $\lfloor n/4 \rfloor$ vertex disjoint mutually crossing elbows.
Additionally, we show a point set that admits no more than $n/3$ mutually crossing elbows.

Finally we study \emph{intersecting families} of graphs, which are not necessarily vertex disjoint. 
A set of edge disjoint graphs with vertices in $S$ is called an \emph{intersecting family} if for any two graphs in the set we can choose an edge in each of them such that they cross. 
We prove a conjecture by Lara and Rubio-Montiel~\cite{lara2019crossing}, namely, that any set $S$ of $n$ points in general position admits a family of intersecting triangles with a quadratic number of elements.

Some other results are obtained throughout this work.
\end{abstract}

\section*{Introduction}
\label{sec:Intro}

All point sets considered here will be assumed to be on the plane and in general position. 
Let $S$ be a set of points on the plane. 
A \emph{geometric graph} on $S$ is a graph whose vertices are points in $S$, and whose edges are line segments joining pairs of points in $S$.
A $k$-path of $S$ is a sequence of $k+1$ different points $s_0, \ldots, s_{k}$ of $S$ together with the edges (line segments) joining $s_i$ to $s_{i+1}$, $i=0, \ldots, k-1$. 
A $k$-cycle of $S$ is a sequence of $k$ points $s_0, \ldots , s_{k-1}$ together with the edges joining $s_i$ to $s_{i+1}$, $i=0, \ldots, k-1$, addition taken $\mod k$.
A $3$-cycle of $S$ is called a triangle.
Two vertex disjoint edges \emph{cross} if they intersect at a point in the interior of both of them.
Two vertex disjoint graphs with vertices in $S$ \emph{cross} if there are two edges, one in each of them, that cross. 
A \emph{Hamiltonian cycle} is a cycle that contains all the points in $S$.

In the study of geometric graphs, we have been mainly interested in studying geometric graphs containing no edges that cross, e.g., triangulations, non-crossing matchings in bicolored point sets, alternating paths, etc.
In this paper we will be interested in finding geometric graphs on point sets such that their edges cross many times, e.g., Hamiltonian cycles whose edges cross as many times as possible. 

We say that a set of points $S$ admits (or contains) a \emph{crossing family of $k$-paths} (respectively \emph{of $k$-cycles}) if there is a set of vertex disjoint paths (cycles) of length $k$ with vertices in $S$, such that any two of them cross. 
Perhaps the first result on crossing families was obtained by Aronov et al.~\cite{aronov1994crossing}; they proved that any set of $n$ points in general position has a \emph{crossing matching} with $\sqrt{n/12}$ edges, i.e. a matching such that any two of its edges cross. 
Pach and Solymoshi~\cite{pach1999halving} proved that if a point set with $n=2m$ points has exactly $m$ halving lines, then it has a crossing perfect matching.
Aronov et al. conjectured in~\cite{aronov1994crossing} that there is always a crossing matching with a linear number of elements; this conjecture remains open. 
Significant progress on this problem was recently achieved by Pach et al.~\cite{pach2021planar}, they proved that any point set has a mutually crossing matching with at least $n^{1-o(1)}$ edges.

In~\cite{aronov1994crossing} the authors also show an example of a point set $S$ such that any crossing matching with vertices in $S$ has at most $n/4$ elements. 
This bound was improved by Evans et al.~\cite{evans2019problems} to $5\lceil n/24\rceil$, and by Aichholzer et al.~\cite{AICHHOLZER2022101899} to $8\lceil n/41\rceil$.

\'Alvarez-Rebollar et al.~\cite{alvarezrebollar2015crossing} proved that any point set admits a crossing family of $3$-paths with $\lfloor n/4\rfloor$ elements, and that for $k\geq 3$, any point set admits a crossing family of $k$-paths with $\lfloor n/(k+1) \rfloor$ elements; these bounds are tight. 
 
An interesting problem on $2$-paths remains open: Is it true that any set of $n$ points contains a crossing set of $2$-paths of linear size?
Using the result of Pach et al.~\cite{pach2021planar} on mutually crossing matchings, it follows easily that any point set $S$ admits a mutually crossing set of $2$-paths with at least $n^{1-o(1)}$.

Using our result on crossing families of $k$-paths it follows easily that any point set has a crossing family of $k$-cycles, $k\geq 3$, with at least $\lfloor n/k \rfloor$ elements, since each $k$-path can be completed to a $k$-cycle.

In~\cite{alvarezrebollar2015crossing}, it was proved that any point set with $n$ elements admits a set of $\lfloor n/6\rfloor$ vertex disjoint crossing triangles; this result was improved by
Fulek et al.~\cite{fulek_2019}, who proved that any point set admits a crossing family of triangles with $\lfloor n/3\rfloor$ elements, which is evidently optimal.

An \emph{orthogonal geometric graph} with vertices on $S$ is a graph whose vertices are points of $S$ and its edges are represented by chains of horizontal and vertical line segments joining pairs of points in $S$; an edge of an orthogonal geometric graph is called an \emph{elbow} if it has one horizontal and one vertical line segment. 
The interested reader can find more on this topic in a survey by Tamassia in~\cite{tamassia1999advances}.

\subsection*{Our results}

In Section~\ref{sec:CrossFam} we prove that any point set has a crossing family with at least $\lfloor n/4 \rfloor$ vertex disjoint elbows, and that there exist point sets in which any crossing family of elbows has at most $n/3$ elements. Then we show that there exists a constant $c$ such that from any family of $n$ mutually crossing triangles, one can always obtain a family of at least $n^c$ mutually crossing $2$-paths (each of which is the result of deleting an edge from one of the triangles) and provide an example that implies that $c$ cannot be taken to be larger than $2/3$. 

In Section~\ref{sec:CrossHamCyc} we address the problem of finding a Hamiltonian cycle that crosses itself as many times as possible.
Observe that by using the result by Aronov et al. on the existence of crossing matchings, it is relatively easy to show that $S$ always admits a Hamiltonian cycle with $O(n \sqrt n)$ crossings. 
To see this, take a crossing matching with $\sqrt{n/12}$ elements, and remove from $S$ the endpoints of this matching. 
Repeat this process inductively on the remaining set of points until we get a set with $O(\sqrt n)$ crossing matchings of size $O(\sqrt n)$. 
By concatenating the edges of these matchings appropriately, we can get a Hamiltonian cycle of $S$ such that its edges cross $O(n \sqrt n)$ times.
In \cite{alvarezrebollar2015crossing} it was proved that $S$ always admits a Hamiltonian cycle such that its edges cross at least $n^2/12-O(n)$ times.
In this paper we show a point set whose Hamiltonian cycles have no more than $5n^2/18-O(n)$ crossings.
It remains as an open problem to close the gap between these bounds.
We point out that if Aronov et al.'s conjecture that any point set has a mutually crossing matching of linear size is true, this result would easily yield a Hamiltonian cycle on $S$ with a quadratic number of crossings.

In Section~\ref{sec:LongMatch} we explore the relation between the number of crossings and the \emph{length} of a matching in a point set, where the length of a matching is the sum of the lengths of its edges.
A question that arises naturally is the following: How many crossings does the longest Hamiltonian cycle, the longest Hamiltonian path, or the longest perfect matching of a point set have?
We show an example of a point set in which the longest perfect matching has no crossings. 
The number of crossings in the other two cases remain as open problems.

Finally, in Section~\ref{intfam} we study intersecting families of triangles on point sets.
Lara and Rubio-Montiel~\cite{lara2019crossing} introduced the concept of \emph{intersecting family}, in which the elements are only required to be edge disjoint, but are allowed to share vertices. 
They proved that any point set admits an intersecting family of $3$-paths, and an intersecting family of triangles of size $O(n^{3/2})$, and conjectured that the true lower bound for both problems is $O(n^2)$. 
We prove that indeed any family of points admits an intersecting family of triangles with $(n^2+3n)/18$ elements.

\section{Crossing families of elbows, 2-paths and triangles}
\label{sec:CrossFam}

In this section we obtain bounds on the number of mutually crossing elbows a point set admits. We also study crossing families of $2$-paths that can be obtained from a collection of crossing triangles, as well as crossing families of simple convex cycles.

\subsection{Matchings using elbows}

In this subsection, we will deal with orthogonal geometric graphs whose edges are elbows. 
Let $S$ be a point set such that no two of its elements lie on the same vertical or horizontal line. 
We now prove:

\begin{thm}
$S$ always admits a family of $\lfloor n/4 \rfloor$ vertex disjoint mutually crossing elbows. There are sets of $n$ points that do not contain families of crossing elbows with more than $n/3$ elements.
\label{elbowsLow}
\end{thm}

\begin{proof}
Suppose without loss of generality that $S$ has $n=4m$ points. 
Let $\mathcal{L}$ be the horizontal line that divides the plane in two regions each containing $2m$ elements of $S$.
Sweep a vertical line $\mathcal M$ from right to left until one of the two regions (to the right of $\mathcal M$) defined by $\mathcal M$ and $\mathcal L$ contains $m$ points of $S$.

Suppose without loss of generality that such a region is above $\mathcal{L}$ and let $\mathcal{B}$ be such a region. 
Let $\mathcal A$ denote the region lying below $\mathcal{L}$ and to the left of $\mathcal{M}$.
Thus, $\mathcal A$ contains at least $m$ points of $S$.
See Figure~\ref{fig:Elbows} (a).

    \begin{figure}[ht]
    	\begin{subfigure}{.49\textwidth}
    		\centering
    		\includegraphics[width=\linewidth]{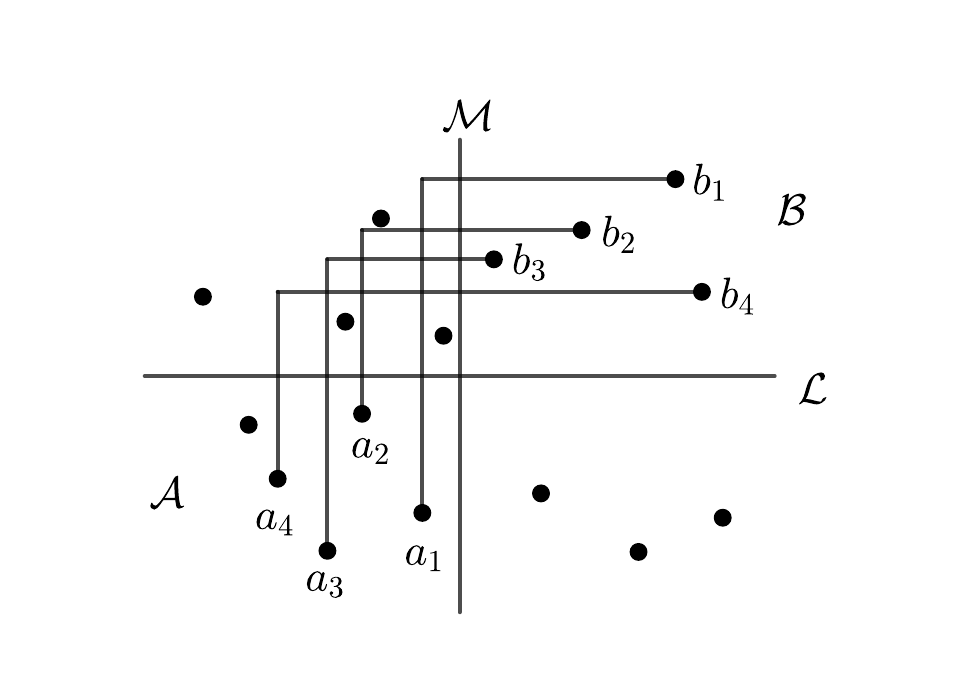}  
    		\caption{Construction of a crossing family of elbows with at least $\lfloor n/4\rfloor$ elements}
    		\label{fig:ElbowsLower}
    	\end{subfigure}
    	\begin{subfigure}{.49\textwidth}
    		\centering
    		\includegraphics[width=\linewidth]{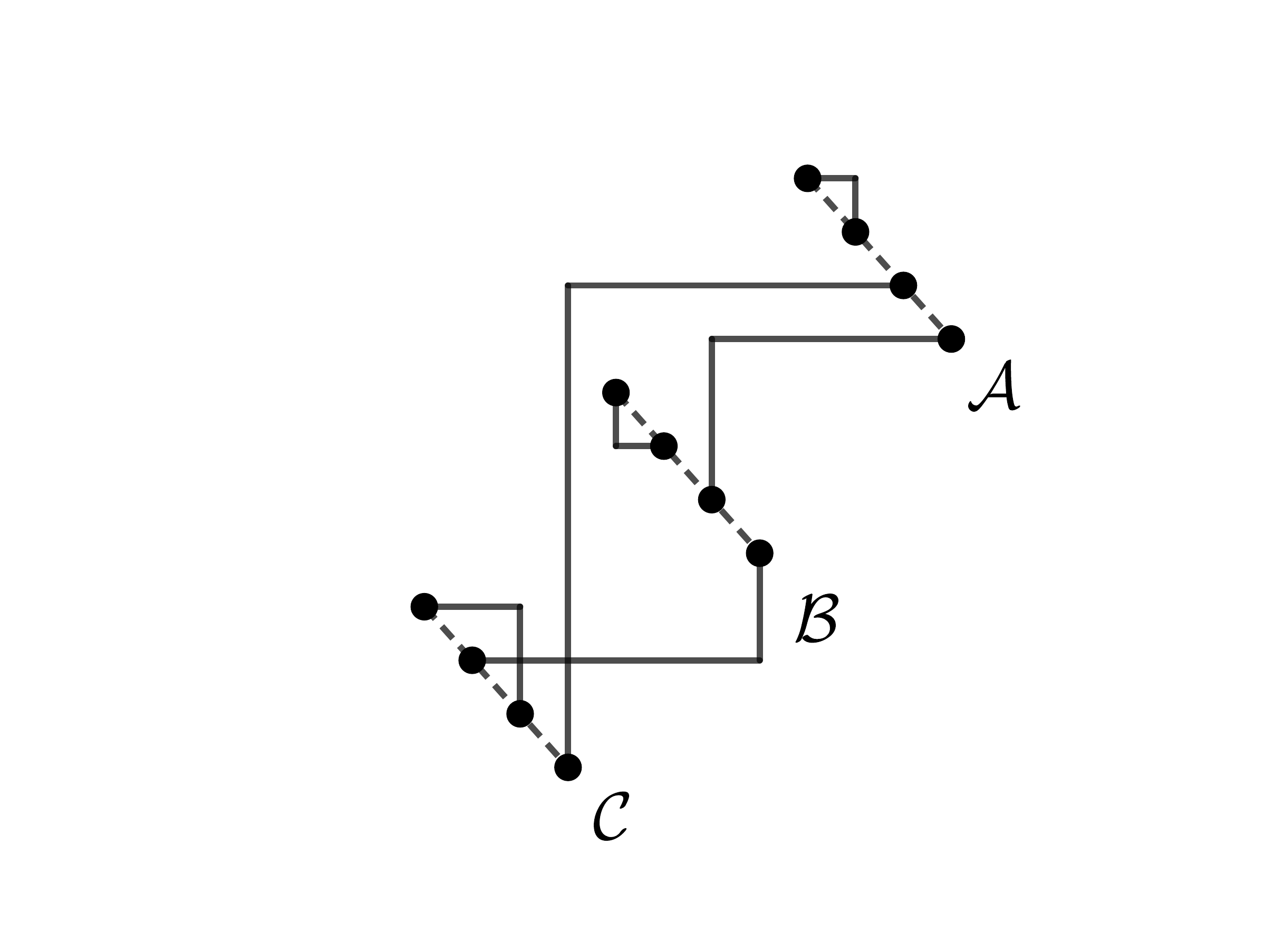}  
    		\caption{Example of a point set with no crossing families of elbows having more than $n/3$ elements}
    		\label{fig:ElbowsUpper}
    	\end{subfigure}
    	\caption{Crossing families of elbows}
    	\label{fig:Elbows}
    \end{figure}

Label $m$ points in $\mathcal A$ by $a_1, \ldots , a_m$ according to their $x$-coordinate from right to left; label the $m$ points in $B$ by $b_1, \ldots, b_m$ according to their $y$-coordinate from top to bottom. 
Join the point $a_i$ with the point $b_i$ with an elbow whose vertical segment contains $a_i$ and whose horizontal segment contains $b_i$. 
It is straightforward to see that these $m=n/4$ elbows mutually cross.

We now prove the upper bound:
     
Let $S$ be a set with $3m$ points such that its elements are divided into three subsets $\mathcal{A}$, $\mathcal{B}$ and $\mathcal{C}$, such that the points of each of these subsets lie on three parallel lines with negative slope. 
Furthermore, assume that the $x$-coordinates of the points of $\mathcal A$ are smaller than the $x$-coordinates of the points in $\mathcal B$, which in turn are smaller than the $x$-coordinates of the points in $\mathcal C$, and that the $y$-coordinates of the points of $\mathcal A$ are smaller than those of the points of $\mathcal B$, which are smaller than those of the points of $\mathcal C$. 
See Figure~\ref{fig:Elbows}(b).
    
Note that any elbow that joins a point of $\mathcal{A}$ to a point of $\mathcal{B}$, and any elbow that joins a point of $\mathcal{B}$ to a point of $\mathcal{C}$ do not cross. 
Observe also that an elbow that joins two points of the same subset do not cross any elbows that join points of the two other subsets.
These two observations imply that, in any crossing family of elbows, there is a subset of $\mathcal A$, $\mathcal B$ or $\mathcal C$ such that all the elbows of the crossing family have an endpoint in such a subset.
Therefore, any crossing family of elbows contains at most $m=n/3$ elements.

\end{proof}

An open problem is that of closing the gap between the lower and the upper bounds of Theorem~\ref{elbowsLow}. We make the following conjecture:
\begin{conj}
There are point sets of $n$ elements that do not admit crossing families of elbows with more than $\lfloor n/4\rfloor$ elements.
\end{conj}

\subsection{2-paths and triangles}
\label{sec:CrossSub}

While trying to establish bounds on the size of crossing families of $2$-paths and matchings that a point set admits, we considered crossing families of triangles and $2$-paths. 
Note that by removing an edge of each triangle we obtain a family of $2$-paths, and by removing an edge of each $2$-path we obtain a matching.
It is straightforward to see that there are crossing families of $2$-paths such that no matter how we remove one edge from each of them, we obtain a set of edges such that no three of them cross each other, see Figure~\ref{fig:2Paths}(a). 
This example can be easily generalized to $k$-paths, $k\geq3$, see Figure~\ref{fig:2Paths}(b).
For crossing families of triangles, we provide an example of a crossing family of $n$ triangles such that any crossing family of $2$-paths that can be obtained by removing one edge from each triangle has at most $O(n^{2/3})$ elements.

    \begin{figure}[ht]
        \centering
    	\begin{subfigure}{.4\textwidth}
    		\centering
    		\includegraphics[width=\linewidth]{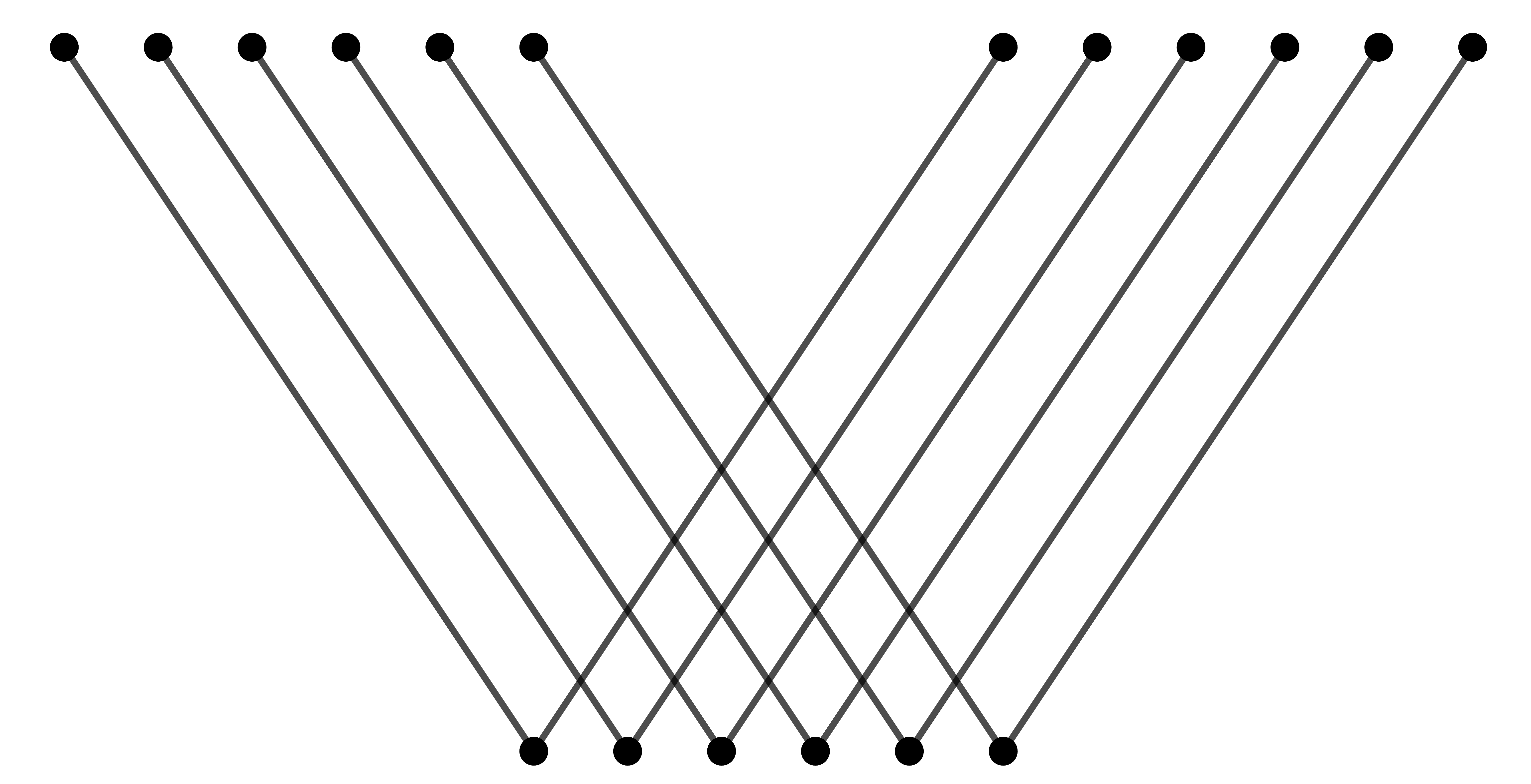}  
    		\caption{A crossing family of $2$-paths in which there are no more than two mutually crossing edges}
    		\label{fig:2Paths_1}
    	\end{subfigure}
    	~%\hspace{}
    	\begin{subfigure}{.4\textwidth}
    		\centering
    		\includegraphics[width=\linewidth]{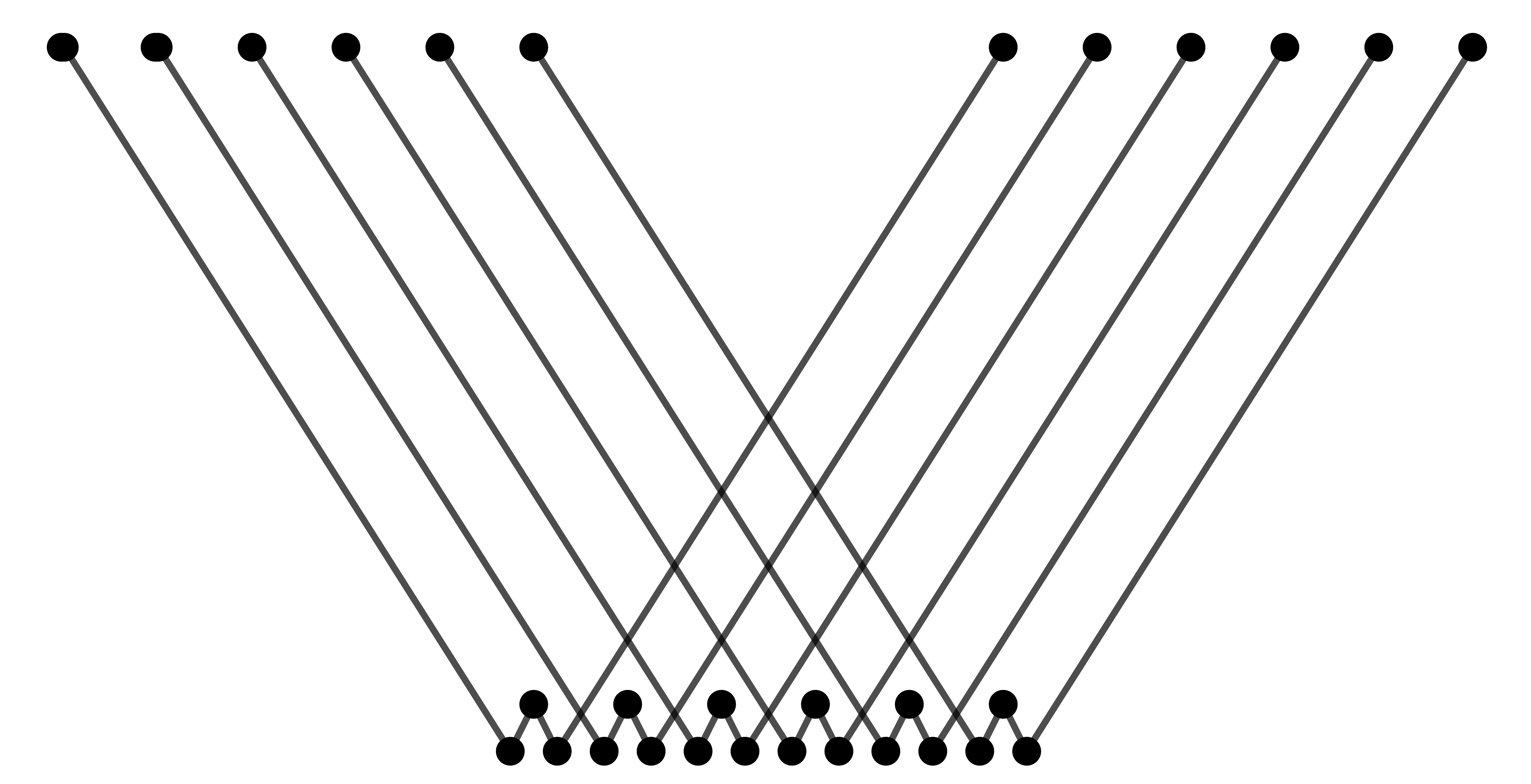}  
    		\caption{A crossing family of $4$-paths in which there are no more than two mutually $3$-paths}
    		\label{fig:2Paths_2}
    	\end{subfigure}
    	\caption{}
    	\label{fig:2Paths}
    \end{figure}

\begin{thm}
There exists a crossing family of $n$ triangles such that, after removing one edge from each of them, in the resulting set of $2$-paths there is no crossing family of size greater than $O(n^{2/3})$.
\label{thm:Triangles}
\end{thm}

\begin{proof}
Assume that $n=m^3$. We construct a family of $n$ crossing equilateral triangles such that one side of each of them is horizontal and the opposite vertex lies below it, see Figure~\ref{fig:exampleConstruction}.
We classify the vertices of each triangle as left, right, and bottom, and the edges as left, right, and top, according to their position. 
Note that if two triangles with parallel sides cross, then one of them contains exactly one vertex of the other one. 

We construct our example in three steps: 
First we draw a set of mutually crossing triangles $T_1,\ldots, T_m$ such that any triangle $T_i$ contains the left vertex of each of the triangles $T_{i+1},\ldots,T_m$, as shown in Figure~\ref{fig:exampleConstruction}(a). 
We now replace each triangle $T_i$ by a set of triangles $T_{i,1},\ldots,T_{i,m}$ such that any triangle $T_{i,j}$ contains the right vertices of the triangles $T_{i,j+1},\ldots,T_{i,m}$, as shown in Figure~\ref{fig:exampleConstruction}(b). 
Finally we replace each triangle $T_{i,j}$ by a set of triangles $T_{i,j,1},\ldots,T_{i,j,m}$ such that any triangle $T_{i,j,k}$ contains the bottom vertices of the triangles $T_{i,j,k+1},\ldots,T_{i,j,m}$, as shown in Figure~\ref{fig:exampleConstruction}(c).
It is straightforward to see that this set of $m^3$ triangles is mutually crossing.

    \begin{figure}[!ht]
    	\begin{subfigure}[t]{.49\textwidth}
    		\centering
    		\includegraphics[width=\linewidth]{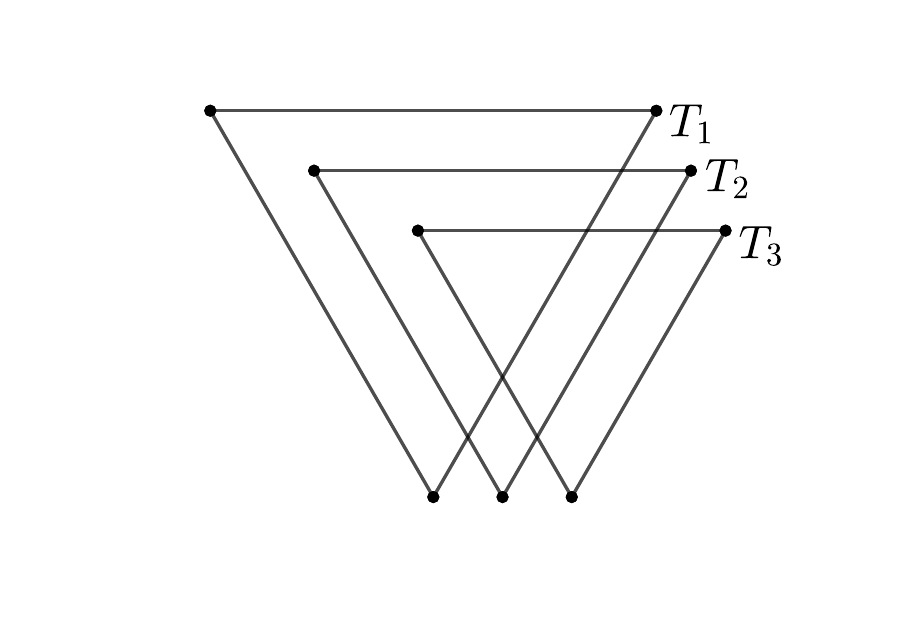}  
    		\caption{First step of the construction}
    		\label{fig:trianglesFirstStep}
    	\end{subfigure}
    	~
    	\begin{subfigure}[t]{.49\textwidth}
    		\centering
    		\includegraphics[width=\linewidth]{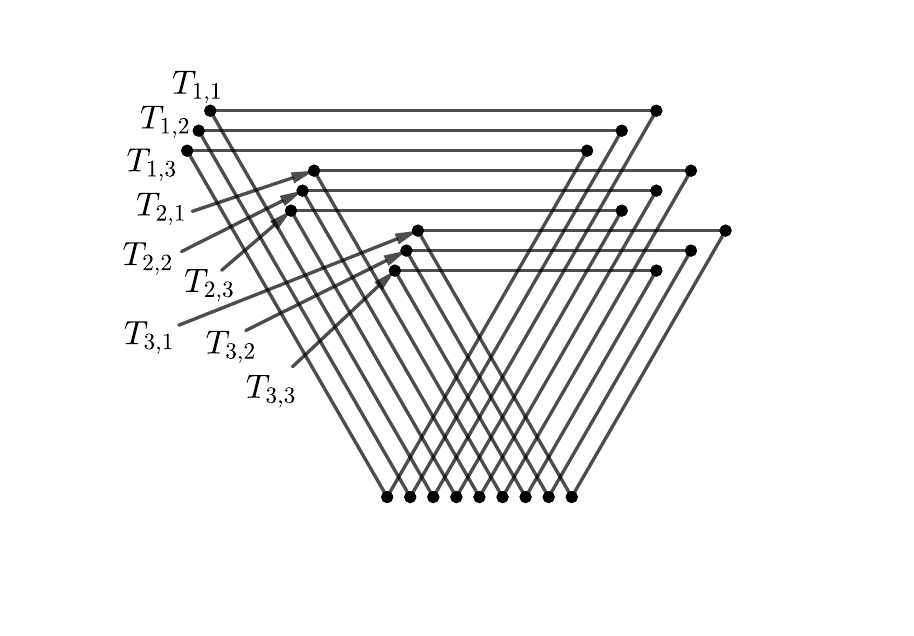}  
    		\caption{Second step of the construction}
    		\label{fig:trianglesSecondStep}
    	\end{subfigure}
    	\begin{subfigure}[t]{\textwidth}
    		\centering
    		\includegraphics[width=0.92\linewidth]{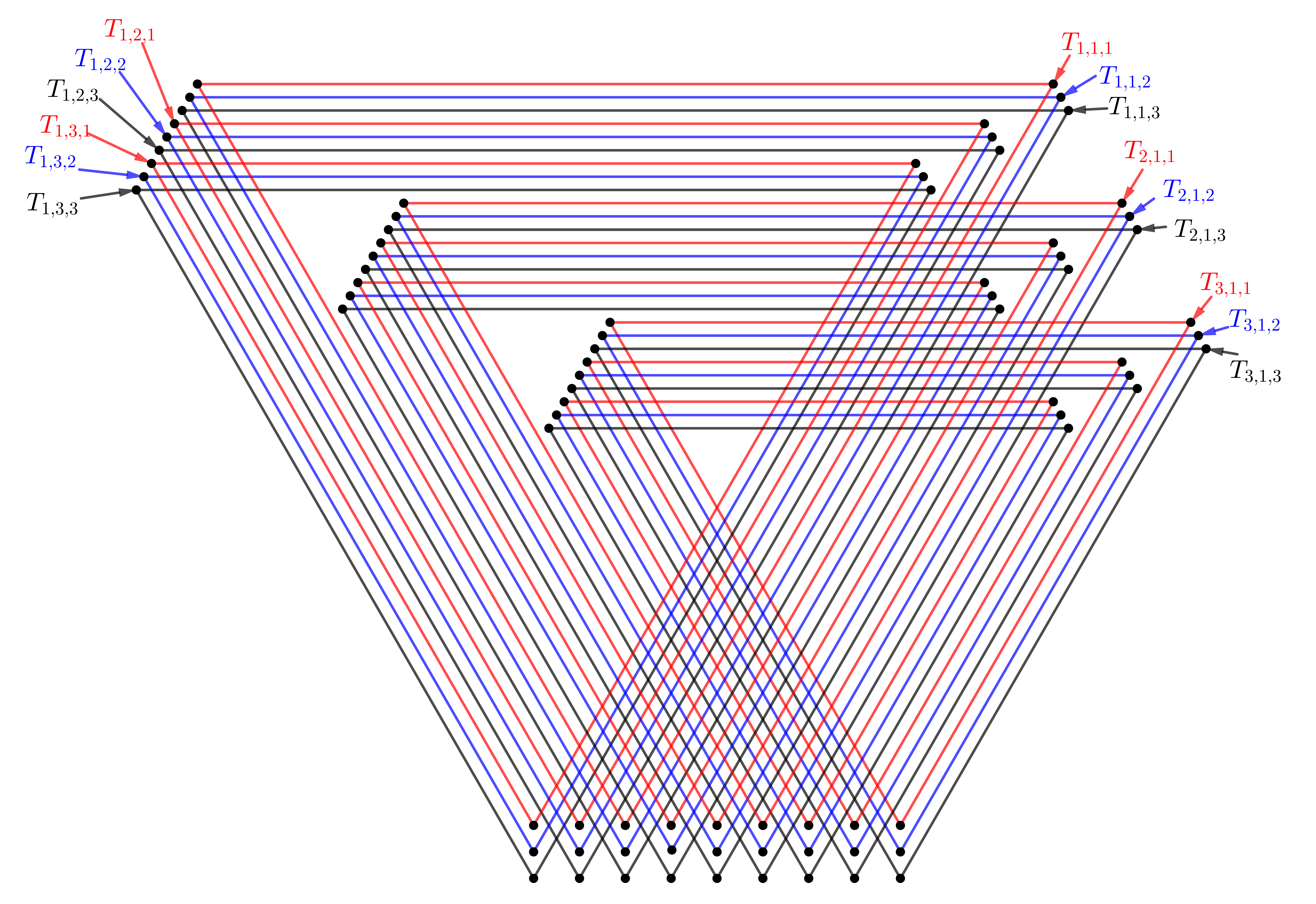}  
    		\caption{Third step of the construction}
    		\label{fig:trianglesThirdStep}
    	\end{subfigure}
    	\caption{Construction of the crossing family of triangles for the proof of Theorem~\ref{thm:Triangles}}
    	\label{fig:exampleConstruction}
    \end{figure}

Suppose now that we have removed one edge from each triangle.
Divide the set of $2$-paths thus obtained into three sets: $\mathcal{L}$, $\mathcal{R}$, and $\mathcal{T}$.
The set $\mathcal{L}$ consists of the $2$-paths resulting from those triangles from which we removed the left edge.
Similarly, $\mathcal{R}$ (respectively $\mathcal{T}$) consists of the $2$-paths resulting from those triangles from which we removed the right edge (respectively the top edge).

By removing the right edge of two triangles $T_{i,j,k}$ and $T_{p,q,r}$, $p\neq i$, we obtain two $2$-paths that do not cross.
In $\mathcal{R}$, any set of mutually crossing $2$-paths comes from a set of triangles with the same first index. 
Therefore, any crossing family of $2$-paths in $\mathcal{R}$ has size at most $m^2$.

After removing the left edge of two triangles $T_{i,j,k}$ and $T_{i,q,k}$, $j\neq q$, we also obtain two $2$-paths that do not cross.
Then, for each pair of indices $(i,k)$, any set of mutually crossing $2$-paths in $\mathcal{L}$ can contain at most one $2$-path that comes from a triangle $T_{i,j,k}$.
Therefore, any crossing family of $2$-paths in $\mathcal{L}$ has size at most $m^2$.

Finally, observe that after removing the top edge of two triangles $T_{i,j,k}$ and $T_{i,j,r}$, $k\neq r$, the two  $2$-paths obtained do not cross. 
Then, for each pair of indices $(i,j)$, any set of mutually crossing $2$-paths in $\mathcal{T}$ can contain at most one $2$-path that comes from a triangle $T_{i,j,k}$.
Therefore, any crossing family of $2$-paths in $\mathcal{T}$ has size at most $m^2$.

Since any mutually crossing set of $2$-paths consists of one mutually crossing set in $\mathcal{R}$, one in $\mathcal{L}$, and one in $\mathcal{T}$, then any crossing family of $2$-paths contains at most $3m^2$ elements, which completes the proof.

\end{proof}

In order to complement the previous result, we show that every family of $n$ mutually crossing triangles induces a family of $n^c$ mutually crossing $2$-paths. 
The proof will require some notions from \cite{alon2005crossing}. 
A \textit{semi-algebraic set} in $\mathbb{R}^d$ consists of all those points which satisfy a given finite Boolean combination of polynomial inequalities and equations in the $d$ variables; furthermore, if these polynomials have degree at most $t$ and there are no more than $t$ polynomial inequalities and equations, then we say that the set has \textit{description complexity} at most $t$.
We say that a graph $G=(V,E)$ has \textit{description complexity at most} $(d,t)$ if the vertices can be represented by a set $V^*$ of $|V|$ points in $\mathbb{R}^d$ such that \[\{(u^*,v^*)\in\mathbb{R}^{2d}:uv\in E\}\] is semi-algebraic of description complexity at most $t$. 
Alon et al. \cite{alon2005crossing} showed that if $G$ has description complexity at most $(d,t)$ then either $G$ or its complement contains (as a subgraph) a complete graph on at least $n^\delta$, where $\delta$ depends only on $d$ and $t$; the following theorem is a consequence of this result.

\begin{thm}
There is a positive constant $c$ such that given any crossing family $\mathcal T$ of $n$ triangles it is possible to remove one edge from each of them so that the resulting collection of $2$-paths contains at least $n^c$ mutually crossing elements.
\end{thm}

\begin{proof}
Label the triangles of $\mathcal T$ as $T_1,T_2,\dots,T_n$ and, w.l.o.g., assume that no two vertices lie on the same horizontal line. 
For each $i$, let $p_i$ be the topmost vertex of $T_i$ and label the edges of $T_i$ as \textit{left (l), right (r)} and \textit{bottom (b)} so that the $l$ edge lies to the left of the $r$ edge, and both these edges are incident to $p_i$. 
Now, we label the edges of the complete graph with vertices $v_1,v_2,\dots ,v_n$ using the following rule: 
For any two indices $i$ and $j$, if removing the $l$ labeled edges of both $T_i$ and $T_j$ results in two crossing $2$-paths, then $v_iv_j$ is labeled with $l$. 
This process is repeated for labels $r$ and $b$.
Note that any edge may receive multiple labels, and an easy case analysis shows that each edge must be labeled at least once. 

Consider the subgraphs $G_l,G_r$ and $G_b$ formed by the edges labeled with $l, r$ and $b$, respectively.
Each of these three graphs, as well as their induced subgraphs, is algebraic with bounded description complexity. 
By the results in \cite{alon2005crossing}, either $G_l$ or its complement contains a clique of size at least $n^\delta$ for some absolute constant $\delta<1$, set $c=\delta^3$. 
If $G_l$ contains a large clique then we may obtain the desired collection of two paths by deleting all edges labelled with $l$ of corresponding triangles, so the interesting case is the one in which the complement of $G_l$ contains a clique with vertex set $V_1$, $|V_1|\geq n^\delta$. 
Let $G_r[V_1]$ denote the subgraph of $G_r$ induced by $V_1$, by repeating the previous argument, we may assume that the complement of $G_r[V_1]$ contains a clique with vertex set $V_2$, $|V_2|\geq |V_1|^\delta\geq n^{2\delta}$. 
Finally, let $G_b[V_2]$ denote the subgraph of $G_b$ induced by $V_2$ and notice that the same argument yields either a clique of size at least $n^{3\delta}$ in $G_b$, in which case we are done, or a vertex set $V_3\subset V_2$ of size at least $n^{3\delta}$ such that no two vertices in $V_3$ are adjacent in any of the graphs $G_l,G_r,G_b$, which is impossible, since each edge of the complete graph received at least one label. 
This concludes the proof.

\end{proof}

\subsection{Crossing simple convex cycles}

A cycle is \textit{simple} if it has no self-intersections. 
A simple cycle is \textit{convex} if its vertices are in convex position. 
We turn our attention towards simple convex cycles. 

\begin{thm}\label{simplecycles}
For every integer $k\geq 4$ there is a constant $c_k$ such that the following holds. 
Every sufficiently large $S$ admits a family of at least $c_kn$ crossing simple convex cycles. 
Furthermore, for $k=4$ we can take $c_4=1/22$.
\end{thm}

In order to show this we require some definitions. 
Given $k$ pairwise-disjoint point sets $\mathcal{P}_1,\mathcal{P}_2,\dots,\mathcal{P}_k$ in the plane, a \emph{transversal of} $(\mathcal{P}_1,\mathcal{P}_2,\dots,\mathcal{P}_k)$ is a collection of points $(p_1,p_2,\dots,p_k)$ such that $p_i\in \mathcal{P}_i$ for every $1\leq i\leq k$. 
Bárány and Valtr \cite{barany1998positive} proved the following.

\begin{thm}\label{fractionalES}
For every integer $k\geq 4$ there is a constant $c_k$ with the following property. 
For every sufficiently large point set $P$ in general position with $m$ elements we can find $k$ disjoint subsets $\mathcal{P}_1,\mathcal{P}_2,\dots,\mathcal{P}_k$ of $P$, each of size $\lceil c_km\rceil$, such that every transversal is in convex position. For $k=4$, one can take $c_4=1/22$.
\end{thm}

Note that if every transversal of $(\mathcal{P}_1,\mathcal{P}_2,\dots,\mathcal{P}_k)$ is in convex position, then each $\mathcal{P}_i$ can be separated from the rest by a line. 
We are ready to prove Theorem \ref{simplecycles}.

\begin{proof}[Proof of Theorem \ref{simplecycles}]
Apply Theorem \ref{fractionalES} to find subsets $\mathcal{P}_1,\mathcal{P}_2,\dots,\mathcal{P}_k$ of $S$ such that each transversal of $(\mathcal{P}_1,\mathcal{P}_2,$ $\dots,\mathcal{P}_k)$ is in convex position. 
We may assume, w.l.o.g., that for every transversal $(p_1,p_2,\dots,p_k)$ the polygon $p_1,p_2,\dots,p_k$ is simple and convex. 
Consider $k$ lines $\mathcal{L}_1,\mathcal{L}_2,\dots,\mathcal{L}_k$ such that $\mathcal{L}_i$ separates $p_i$ from the rest of the subsets. 
We order the points in each $\mathcal{P}_i$ according to their distance to $\mathcal{L}_i$, from closest to furthest ($\mathcal{L}_i$ can be chosen so that no two points of $S$ are at the same distance from $\mathcal{L}_i$); we denote them by $p_{i,1},p_{i,2},\dots,p_{i,\lceil c_kn\rceil}$, according to this order. 
We construct $\lceil c_kn\rceil$ crossing simple cycles, as follows: For $1\leq i\leq \lceil c_kn\rceil$, the $i$-th cycle is given by $p_{1,\lceil c_kn\rceil-i+1},p_{2,i},p_{3,i},p_{4,i},\dots,p_{k,i}$. 
We prove that these cycles are mutually crossing. 
We draw a line that is parallel to $\mathcal{L}_k$ through each element of $\mathcal{P}_k$, and then repeat this process for $\mathcal{L}_1$ and $\mathcal{P}_1$, $\mathcal{L}_2$ and $\mathcal{P}_2$. 
Each cycle contains a $2$-path going through $\mathcal{P}_k,\mathcal{P}_1$ and $\mathcal{P}_2$, in that order. 
If any two of the aforementioned paths do not cross, then their vertices need be arranged as in Figure \ref{fig:Cycle} (b). 
Since the cycles containing these paths are convex and the rest of their vertices lie in the shaded region, they must cross each other. 
The result follows.

\end{proof}

    \begin{figure}[ht!]
        \centering
    	\begin{subfigure}[t]{.4\textwidth}
    		\centering
    		\includegraphics[width=\linewidth]{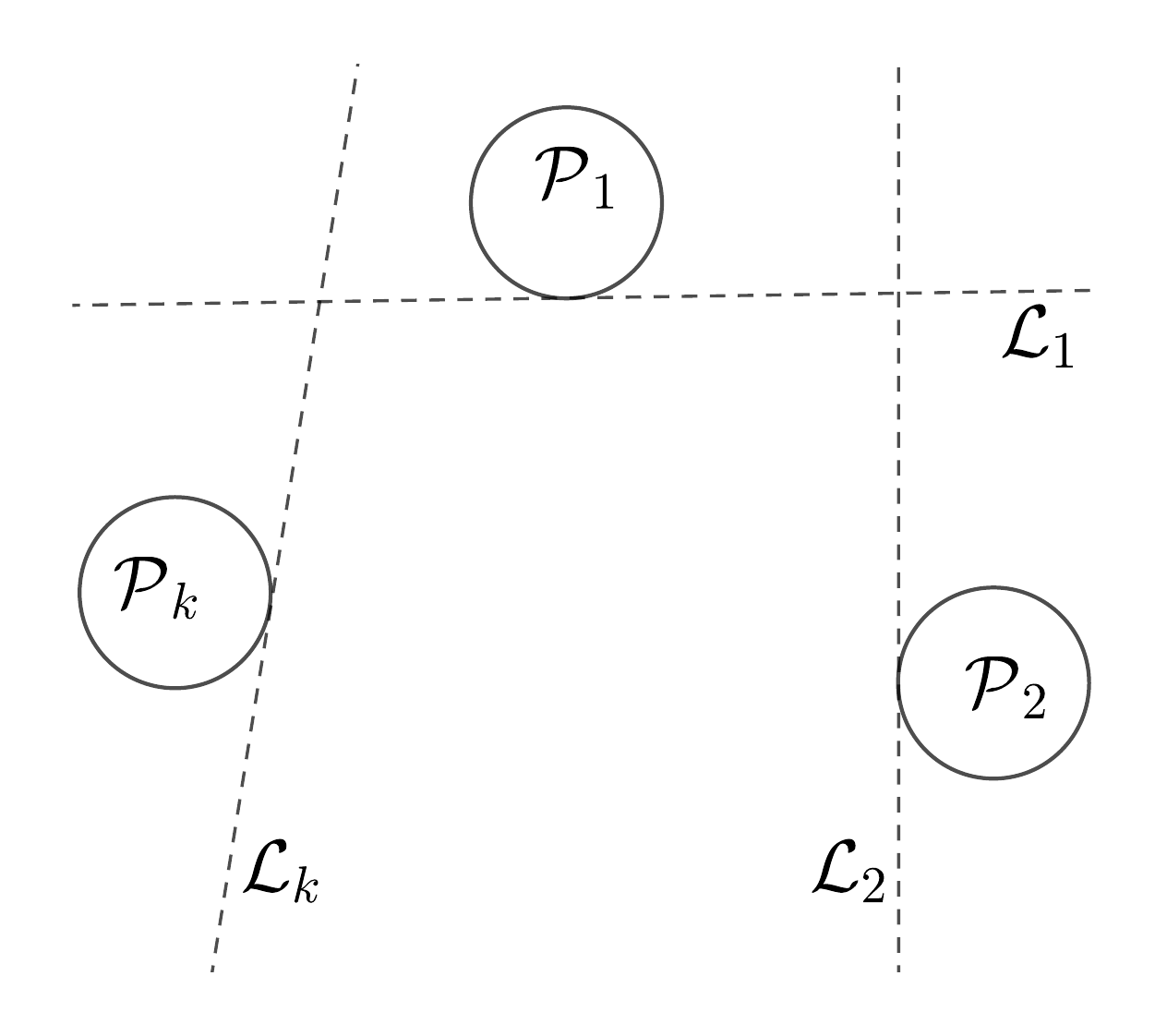} 
    		\caption{The sets $\mathcal{P}_1$, $\mathcal{P}_2$ and $\mathcal{P}_3$ and the lines $\mathcal{L}_1,\mathcal{L}_2$ and $\mathcal{L}_3$}
    		\label{fig:Cycle1}
    	\end{subfigure}
    	~
    	\begin{subfigure}[t]{.4\textwidth}
    		\centering
    		\includegraphics[width=\linewidth]{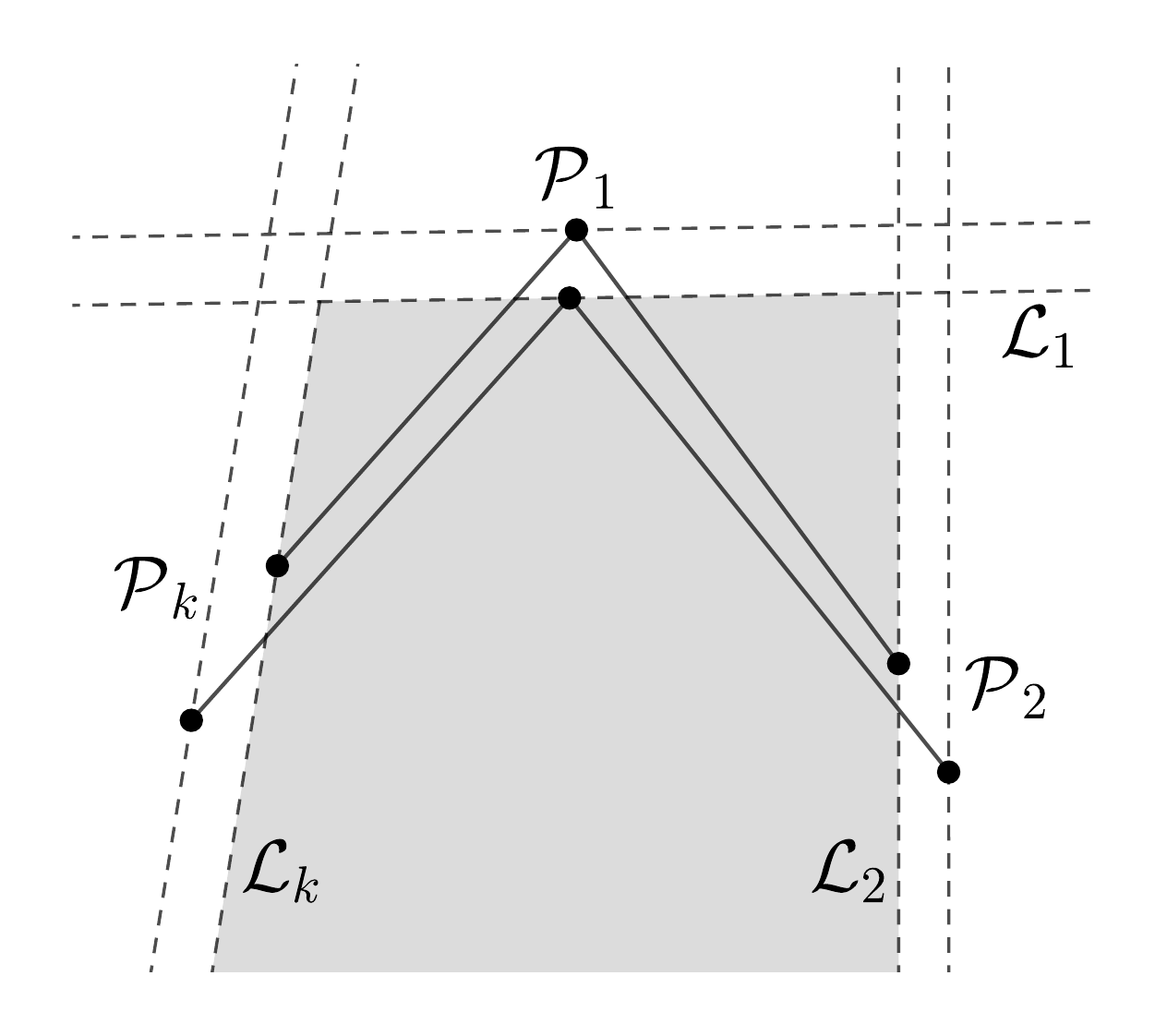} 
    		\caption{Example of a configuration in which two $2$-paths with points in $\mathcal{P}_k$, $\mathcal{P}_1$ and $\mathcal{P}_2$ do not cross each other %How the configuration must look like when the $2$-paths do not cross each other.
    		}
    		\label{fig:Cycle2}
    	\end{subfigure}
    	\caption{}
    	\label{fig:Cycle}
    \end{figure}

Remark: The proof of Theorem \ref{fractionalES} in \cite{barany1998positive} relies on the celebrated Erdős-Szekeres theorem for point sets and the so called \textit{same-type lemma}. 
The asymptotic bounds of the same-type lemma were improved by Fox et al. \cite{fox2016polynomial}. 
For $k=4$, Pór \cite{poor2003partitioned} proved a partition version of Theorem \ref{fractionalES}.

\section{Crossings in hamiltonian cycles}
\label{sec:CrossHamCyc}
 
Let $S$ be a set of $n$ points, and $C_S$ a Hamiltonian cycle on $S$, that is a cycle that visits all the points of $S$. 
Two edges of a graph are called \emph{incident} if they have a common vertex. 
The next result gives a sharp upper bound on the maximum number of crossings any such cycle has; the case when $n$ is odd, was presented as Theorem 4 in \cite{alvarezrebollar2015crossing}, the case when $n$ is even is harder to prove. 
For the sake of completeness, we also include the proof when $n$ is odd.

\begin{thm}
The maximum number of times the edges of $C_S$ cross is  
$n(n-3)/2$ if $n$ is odd or $n(n-4)/2+1$ if $n$ is even. 
For every $n$, there are examples of point sets achieving these bounds.
\end{thm}

\begin{proof}
Let $e$ be an edge in $C_S$. 
Observe that $e$ does not cross the two edges in $C_S$ incident to $e$, thus crosses at most $n-3$ edges of $C_S$. 
This bound is achieved for $n=2m+1$; take a set of $2m+1$ points on a circle numbered in the clockwise order with the integers $0, \ldots, 2m$, and join each point $k$ to the points $k+m$ and $k-m$, addition taken $\mod 2m+1$. 

The even case is more intricate. 
Let $S$ be a point set with $n=2m$ elements and consider a Hamiltonian cycle $C_S$ on $S$. 
A pair of edges of $C_S$ is said to be \emph{avoiding} if they are not incident and do not cross. We will show that any Hamiltonian cycle on $S$ has at least $m-1$ avoiding pairs of edges.
Our proof proceeds by induction on $m$. 
Our result is clearly true for $m=2$.

Classify the edges of $C_S$ into two classes according to the number of avoiding pairs to which they belong: 
$E_1$ will contain the edges that belong to at most one avoiding pair of edges, and $E_2$ those edges that belong to at least two avoiding pairs. 

For every edge in $E_2$, assign to each of its incident edges one avoiding pair containing it; different edges are assigned different avoiding pairs. 
If every edge of $C_S$ belongs to or has been assigned an avoiding pair, the number of such pairs must be at least $n/2=m$ and we are done, since \[\frac{n(n-3)}{2} - \frac{n}{2}= \frac{n(n-4)}{2}.\]

Assume, then, that at least one edge $e$ of $C_S$ has not been assigned an avoiding pair.
Observe that $e$ is not incident to any edge in $E_2$, for otherwise it would have been assigned an avoiding pair. 
Let $\mathcal{L}(e)$ be the line supporting $e$, and $e_1$ and $e_2$ be the two edges of $C_S$ that are incident to $e$ (observe they both belong to $E_1$). 
Since $e$ crosses every edge of $C_S$ other than $e_1$, $e_2$ and itself, $\mathcal{L}(e)$ must have $m-1$ points of $S$ on each side, and $e_1$ and $e_2$ lie on different sides of $\mathcal{L}(e)$. 
Thus, $e_1$ and $e_2$ form an avoiding pair. Let $x,y,z,w$ be the points such that $e_1=xy$,  $e=yz$ and $e_2=zw$; we show by contradiction that these four points lie in convex position.
Assume, w.l.o.g. that point $y$ is contained in the triangle with vertices $x,z,w$, and let $wp$ be the edge of $C_S$ distinct from $e$ that is incident to $e_2$, see Figure~\ref{fig:ConvexCycleUpper}. 
Since $e$ belongs to no avoiding pairs, $wp$ intersects $yz$. 
However this implies that $xy$ and $wp$ form an avoiding pair, and that $xy$ and $wz$ also form an avoiding pair, and thus $e_1$ belongs to $E_2$, a contradiction. 
Thus, $x,y,z,w$ are in convex position.

\begin{figure}[ht!]
		\centering
		\includegraphics[width=.6\textwidth]{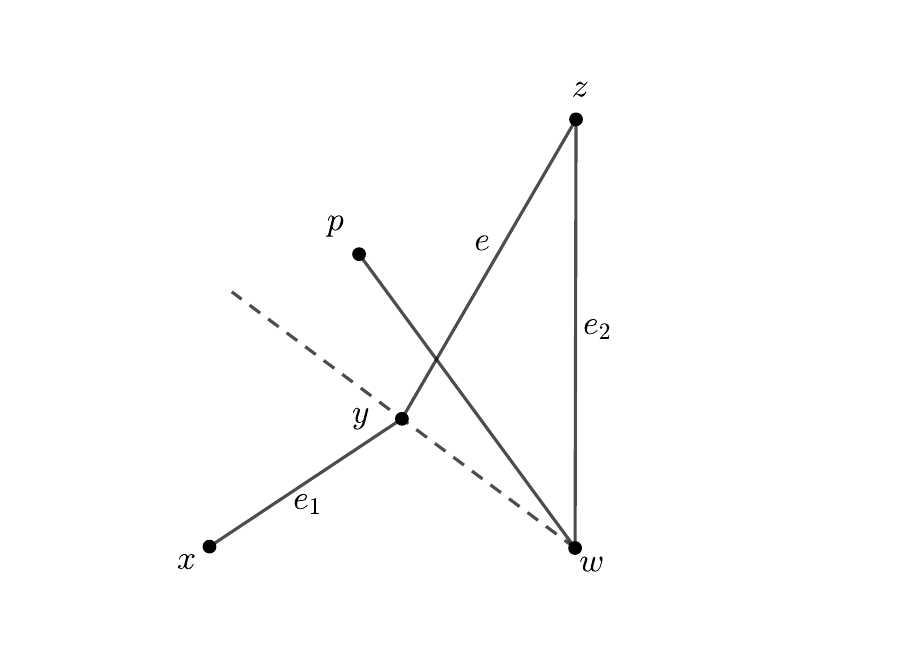}  
		\caption{If points $x$, $y$, $z$ and $w$ were not in convex position, then the edge $e_1$ would belong to at least $3$ avoiding pairs}
		\label{fig:ConvexCycleUpper}
\end{figure}

Now remove from $C_S$ the edges $e_1$, $e$ and $e_2$, and insert the edge $e'=xw$.
This results in a Hamiltonian cycle $C'_S$ on a set of $n-2=2(m-1)$ points which, by induction, contains at least $m-2$ avoiding pairs of edges.
Notice that every edge that crosses $e_1$ and $e_2$ must also cross $e'$. 
This implies that every pair of avoiding edges in $C'_S$ that includes $e'$ corresponds to a pair of avoiding edges in $C_S$ that contains either $e_1$ or $e_2$. 
Since $e_1$ and $e_2$ form an avoiding pair, the number of avoiding pairs of edges in $C_S$ is at least $(m-2)+1=m-1$, as desired. 
Since $n(n-3)/2-(m-1)=n(n-4)/2+1$, and our result follows. 

We proceed now to obtain point sets with an even number of elements for which our bound is attained. 
We construct our examples recursively. 
We will show that if $n$ is even then for any point set in convex position there is a Hamiltonian cycle with $n(n-4)/2+1$ edge crossings. 
Moreover, the cycle can be chosen such that it has an edge $f$ which crosses $n-3$ edges and the two edges incident to $f$ lie one on each of the two semi-planes defined by the line containing $f$.

For $n=4$ take any four points in convex position and construct a cycle with one crossing. 
Let $f$ be any of the edges that cross.
Suppose now that for a set $S$ with $n$ points in convex position, $n$ even, we have constructed a Hamiltonian cycle $C''_S$ such that its edges cross $n(n-4)/2+1$ times and has an edge $f$ that crosses $n-3$ edges. 
Suppose that the vertices of $f$ are labelled $a$ and $b$, see Figure~\ref{fig:ConvexCycle}(a). 
Add two points $p$ and $q$ close enough to $S$ in such a way that $S \cup \{p,q\}$ is in convex position, see Figure~\ref{fig:ConvexCycle}(b). 
We now remove $f$ and replace it with three edges $f_1, f_2,f_3$, see Figure~\ref{fig:ConvexCycle}(b).

    \begin{figure}[ht!]
    	\begin{subfigure}[t]{.49\textwidth}
    		\centering
    		\includegraphics[width=\linewidth]{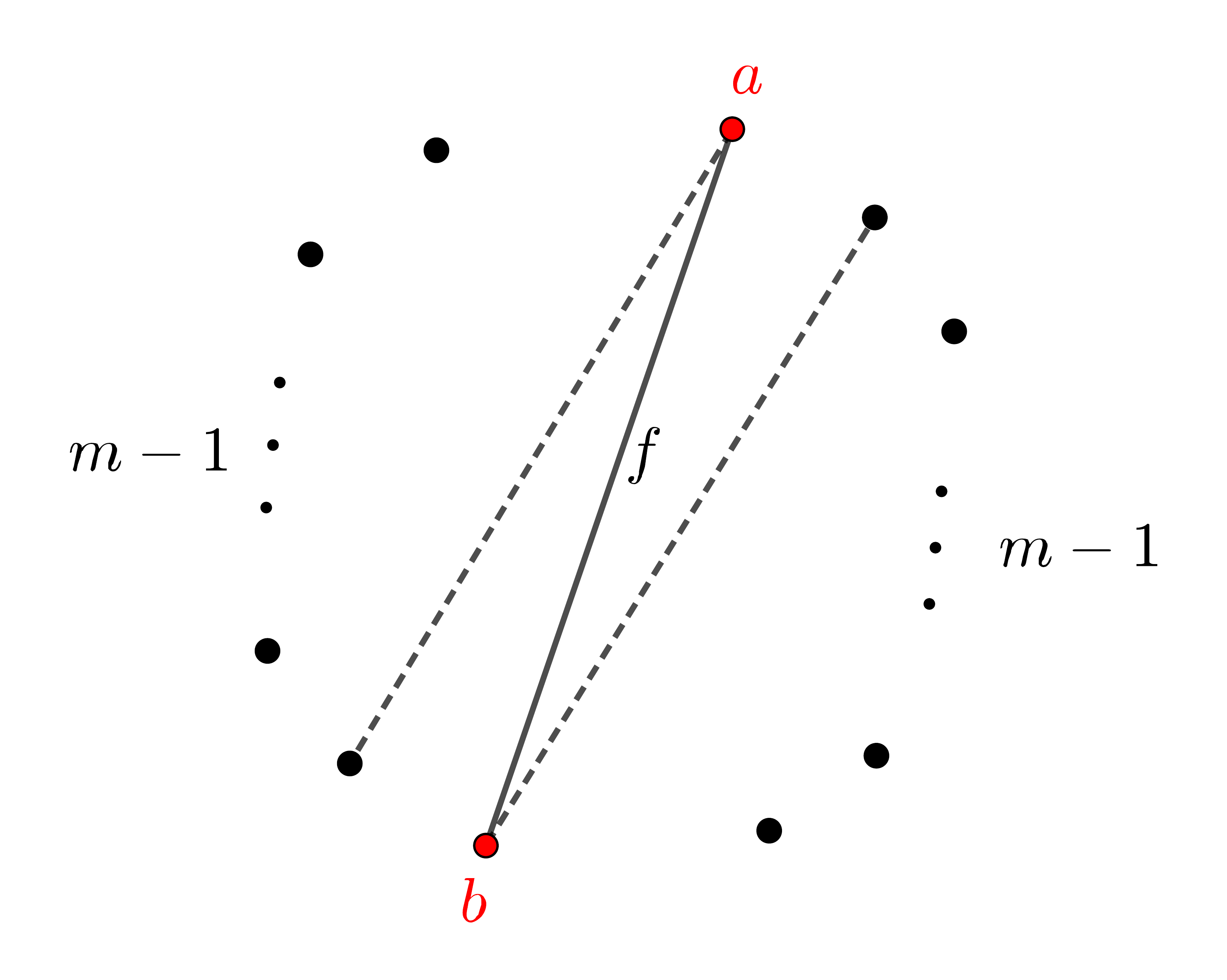}  
    		\caption{An edge crossing other $n-3$ edges before the substitution}
    		\label{fig:ConvexCycle1}
    	\end{subfigure}
    	~
    	\begin{subfigure}[t]{.49\textwidth}
    		\centering
    		\includegraphics[width=\linewidth]{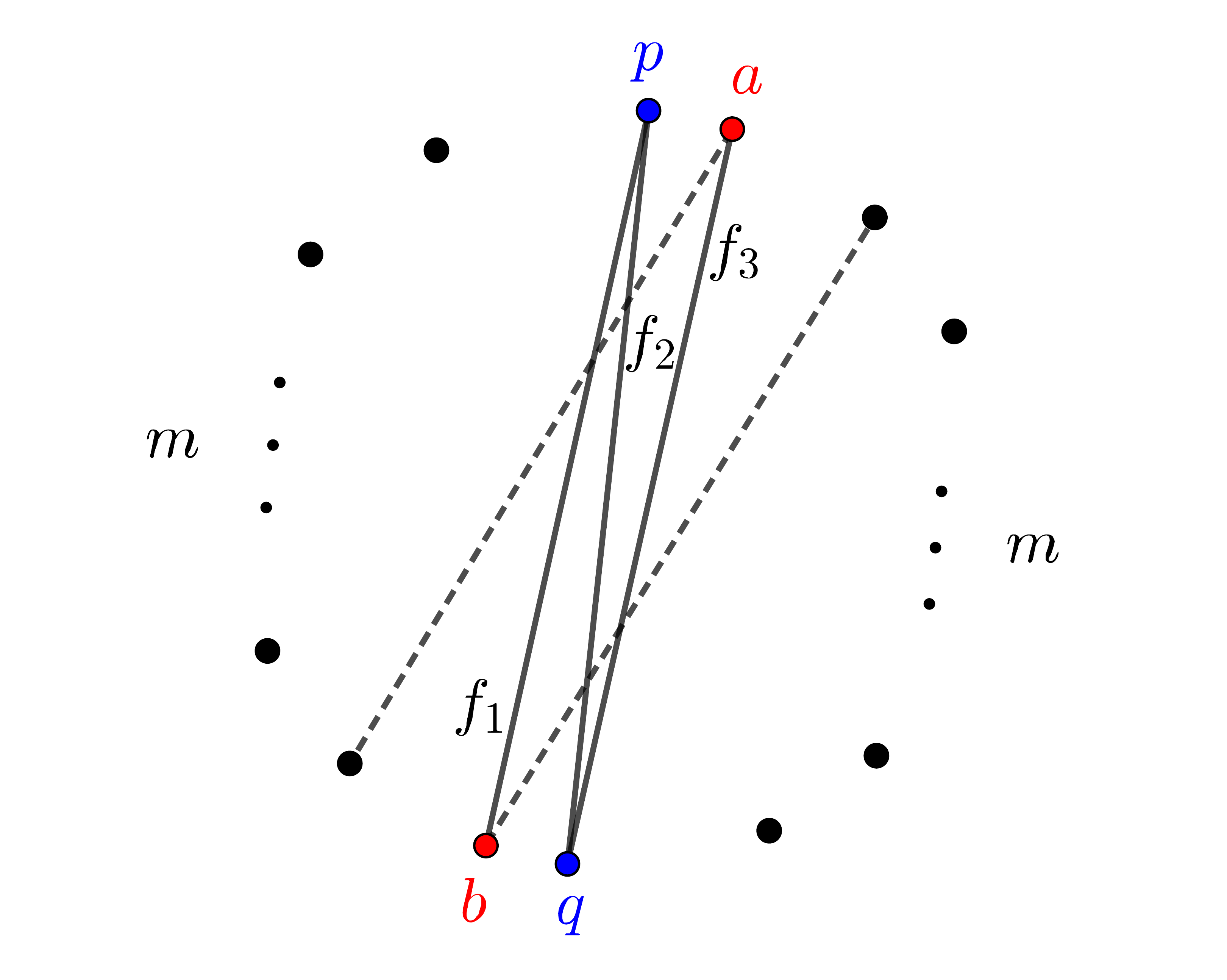}  
    		\caption{An edge crossing other $n-3$ edges after the substitution }
    		\label{fig:ConvexCycle2}
    	\end{subfigure}
    	\caption{}
    	\label{fig:ConvexCycle}
    \end{figure}

This way, we obtain an example with $n+2$ points such that each of $f_1$ and $f_3$ crosses $(n-3)+1=n-2$ edges, and $f_2$ crosses $(n-3)+2=n-1$ edges. 
Thus the total number of crossings is now \[\frac{n(n-4)}{2}+1+2(n-2)+(n-1)-(n-3)=\frac{(n+2)(n-2)}{2}+1.\] Observe that $f_2$ intersects $n-1=(n+2)-3$ edges, and the edges incident to it lie on different semi-planes determined by the line through $p$ and $q$.
This completes our proof.

\end{proof}

\subsection{Minimum number of crossings for Hamiltonian cycles}
In~\cite{alvarezrebollar2015crossing} the problem of finding a Hamiltonian cycle crossing itself as many times as possible was addressed. 
There, it was shown that every point set $S$ of $n$ points in the plane admits a Hamiltonian cycle with at least $n^2/12-O(n)$ crossings.

We present now a point set $S$ with $n=3m$ points such that any Hamiltonian cycle on $S$ has at most $5n^2/18 - O(n)$ crossings, which together with the previous result gives us the following:

\begin{thm}
Every point set admits a Hamiltonian cycle with at least $n^2/12 - O(n)$ crossings. 
Furthermore, there is a point set $S$ with $n=3m$ points such that any Hamiltonian cycle on $S$ has at most $5n^2/18 - O(n)$ crossings.
\end{thm}

\begin{proof}
Consider three convex curves $\mathcal A$, $\mathcal B$ and $\mathcal C$, as shown in Figure~\ref{fig:Blades}.
Let $S$ be a set of $n=3m$ points that has $m$ points on each of the curves $\mathcal A$, $\mathcal B$ and $\mathcal C$.
Let $C_S$ be a Hamiltonian cycle on $S$.
Let $a$ be the number of edges of $C_S$ that join two points on $\mathcal A$.
Similarly, let $b$ (respectively $c$) be the number of edges of $C_S$ that join two points on $\mathcal B$ (respectively on $\mathcal C$).
Let $d$ be the number of edges of $C_S$ that joins a point on $\mathcal A$ to a point on $\mathcal B$.
Similarly, let $e$ (respectively $f$) be the number of edges of $C_S$ that join a point on $\mathcal B$ to a point on $\mathcal C$ (respectively a point on $\mathcal C$ to a point on $\mathcal A$).

    \begin{figure}[ht!]
		\centering
		\includegraphics[width=.6\textwidth]{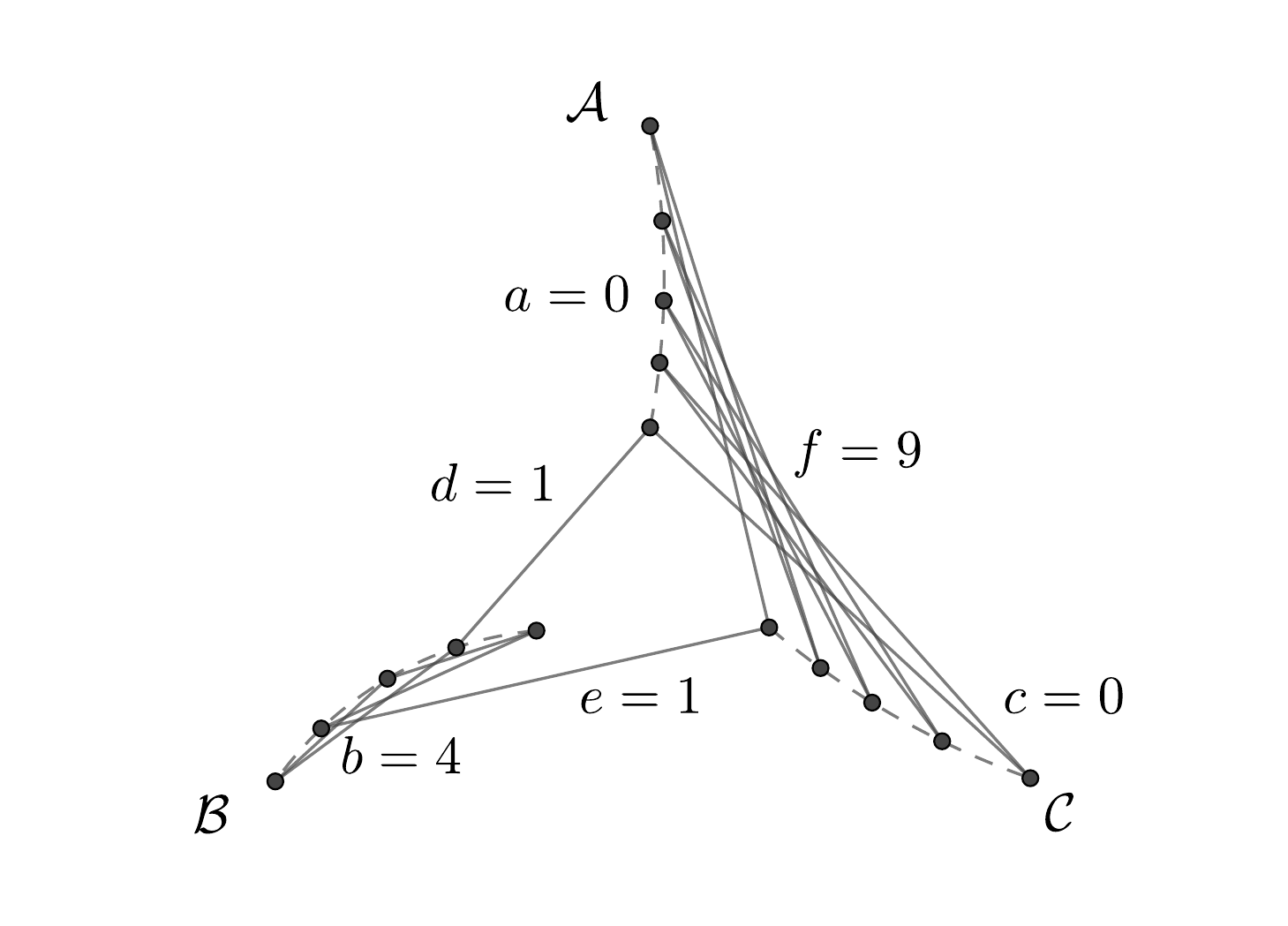}  
		\caption{A point set whose Hamiltonian cycles contain no more than $5n^2/18-O(n)$ crossings.
		A Hamiltonian cycle with $5n^2/18-O(n)$ crossings is depicted
		}
		\label{fig:Blades}
    \end{figure}

Observe that an edge that joins two points on $\mathcal A$ or that joins a point on $\mathcal A$ to a point on $\mathcal B$ can only cross edges joining two points on $\mathcal A$ or joining a point on $\mathcal A$ to a point on $\mathcal B$.
The analogous situation occurs to the edges that join two points on $\mathcal B$ or a point on $\mathcal B$ to a point on $\mathcal C$ (respectively two points on $\mathcal C$ or a point on $\mathcal C$ to a point on $\mathcal A$).
Therefore, the number of crossings of $C_S$ is at most:
\begin{equation}
\label{crossings}
\frac{(a+d)^2+(b+e)^2+(c+f)^2}{2}-O(n)
\end{equation}
As all the vertices of $C_S$ have degree two, the following three constraints hold:
\[
\begin{array}{c}
d + 2a + f = 2m \\
e + 2b + d = 2m \\
f + 2c + e = 2m
\end{array}
\]
Hence, we can express $d$, $e$ and $f$, in terms of $a$, $b$, $c$, and $m$ as follows:
\[
\begin{array}{c}
d=m-a-b+c \\
e=m+a-b-c \\
f=m-a+b-c 
\end{array}
\]
Substituting these expressions in~\eqref{crossings}, the number of crossings of $C_S$ is at most:
\[
\frac{(m-b+c)^2+(m+a-c)^2+(m-a+b)^2}{2}-O(n)=\frac{3}{2}m^2+a^2+b^2+c^2-ab-ac-bc-O(n)
\]
Since this expression is symmetric with respect to $a$, $b$ and $c$, we can suppose without loss of generality that $0\leq a\leq b\leq c\leq m$. 
Using this supposition, we have $a^2\leq ab$ and $b^2\leq bc$, and thus $a^2+b^2-ab-ac-bc\leq 0$. 
Using also that $c^2\leq m^2$, the following holds: 
\[\frac{3}{2}m^2+a^2+b^2+c^2-ab-ac-bc\leq \frac{3}{2}m^2+c^2\leq\frac{5}{2}m^2=\frac{5}{18}n^2.\]
Therefore, the number of crossings that a Hamiltonian cycle in this point set can have is at most $5n^2/18-O(n)$.
\end{proof}

In fact, the point set given in the previous proof has a Hamiltonian cycle with $5n^2/18-O(n)$ crossings. See Figure~\ref{fig:Blades}.

\section{Longest perfect matching having no crossings}
\label{sec:LongMatch}

It is well known that a perfect matching of a set of $2m$ points that minimizes the sum of the lengths of its edges, has no crossing edges. 
This might leads us to think that a longest matching, i.e. the one that maximizes the sum of the lengths of its edges has many crossings; we will show now that that is not necessarily the case. 
We prove the following result.

\begin{thm}
There exist point sets $S$ on the plane whose longest perfect matching has no crossings.
\end{thm}

\begin{proof}
To obtain $S$, we will use a set of segments obtained by Villanger as presented in a paper by Tverberg~\cite{tverberg1979seperation}.
We describe the point set as follows.

Consider $m$ straight line segments $s_1, \ldots, s_m$ such that for each $k\geq 3$, $s_k$ intersects the convex hull of $s_i\cup s_j$, $1\leq i<j<k\leq m$, see Figure~\ref{fig:Villanger}.

	\begin{figure} [ht]
 		\centering
 		\includegraphics[width=0.5\textwidth]{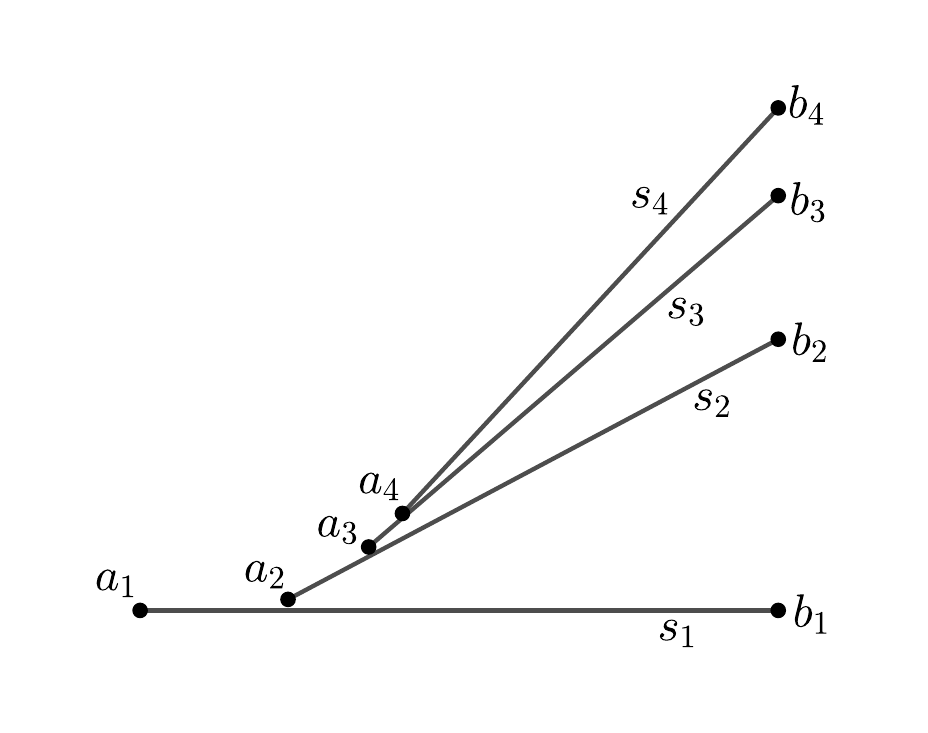}
 		\caption{Villanger's configuration}
		\label{fig:Villanger}
	\end{figure}
		
Let $S(s_i)$ denote the slope of the segment $s_i$, $1\leq i\leq m$. 
The set of segments is constructed in such a way that $0=S(s_1)<S(s_2)<\cdots<S(s_m)<\pi/2$ and, given $\varepsilon>0$, the left endpoint of each segment $s_i$ is at distance less than $\varepsilon$ from the left endpoint of $s_1$.

For each segment $s_i$, $1\leq i\leq m$, let $a_i$ be its left endpoint and $b_i$ be its right endpoint.
Let $\mathcal A=\{a_1,\ldots, a_m\}$, $\mathcal B=\{b_1,\ldots,b_m\}$, and $S=\mathcal A\cup \mathcal B$. 
The segments are such that the distance between any point in $\mathcal A$ and any point in $\mathcal B$ is greater than the distance between any pair of points in $\mathcal A$ or the distance between any pair of points in $\mathcal B$. 
This forces the longest matching to contain only edges joining points in $\mathcal A$ to points in $\mathcal B$.
The points of $\mathcal{B}$ are constructed in such a way that they all lie roughly on the same vertical line.

Let $a_p$ and $a_q$ be two points in $\mathcal A$, and $b_r$ be a point in $\mathcal B$.
Let $H_{p,q,r}$ denote the hyperbola that passes through $b_r$ and whose foci are $a_p$ and $a_q$, see Figure~\ref{fig:Hyperbola}.
The point set $S$ is constructed in such a way that, for $2\leq s \leq m$, the point $b_s$ lies above the right branch of all hyperbolas $H_{p,q,r}$, where $p<q\leq s$ and $r\leq s-1$.

Let us observe an implication of the last condition:
Take two points $a_p$ and $a_q$ in $\mathcal A$ with $p<q$, and two points $b_r$ and $b_s$ in $\mathcal B$ with $r<s$. 
Let $d(a,b)$ denote the euclidean distance between points $a$ and $b$.
Then, $d(b_s, a_p)-d(b_s, a_q) < d(b_r, a_p)-d(b_r, a_q)$ because $b_s$ is above the right branch of hyperbola $H_{p,q,r}$, see Figure~\ref{fig:Hyperbola}.
Then $d(b_s,a_p)+d(b_r,a_q)<d(b_s,a_q)+d(b_r,a_p)$, and thus, the matching $\{a_p b_s,a_q b_r\}$ is shorter than the matching $\{a_p b_r,a_q b_s\}$ on these four points.

    \begin{figure}[ht]
 		\centering
 		\includegraphics[width=0.4\textwidth]{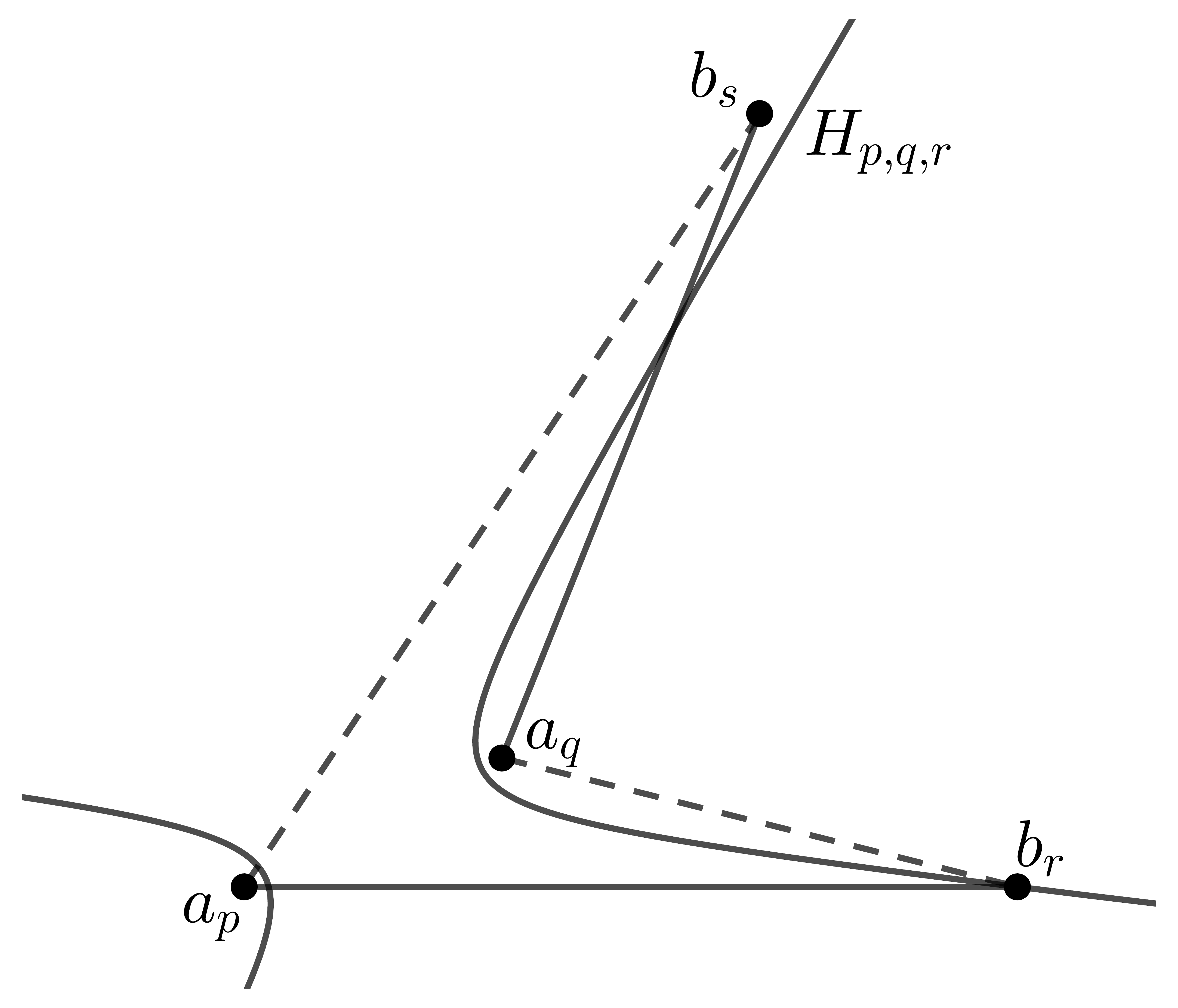}
 		\caption{Four-point configuration in which $d(a_p,b_s)+d(a_q,b_r)<d(a_q,b_s)+d(a_p,b_r)$}
		\label{fig:Hyperbola}
    \end{figure}

We now prove that the matching $\{s_1,\ldots,s_m\}$ is in fact the longest perfect matching on $S$. 
We refer to a perfect matching in which each of its segments joins a point in $\mathcal A$ to a point in $\mathcal B$ as a \emph{potential matching}.
As noted before, the longest matching must be a potential matching and for the proof we consider only potential matchings.

We define a bijection between the set of potential matchings and the set of all permutations on $[m]=\{1, \ldots, m\}$ as follows.
To the potential matching $\{a_1 b_{i_1},\ldots, a_m b_{i_m}\}$ corresponds the permutation $\{i_1,\ldots, i_m\}$.
In this manner, if the segment $a_i b_j$ belongs to a potential matching, then the number $j$ is in the position $i$ in the corresponding permutation. 
Therefore, we will prove that the permutation corresponding to the longest matching is the identity permutation $Id=\{1,\ldots,m\}$. 
To do this, we consider transpositions in which the matching corresponding to the permutation before the transposition is longer than the matching corresponding to the permutation after the transposition.
We say that such transpositions \emph{reduce the length of the matching}.

Consider a permutation $\Sigma=\{\sigma_1,\ldots,\sigma_m\}$ on $[m]$.
Let $k=\sigma_i$ and $\ell=\sigma_j$ be two elements in $\Sigma$ such that $i<j$, i.e. $a_i$ is joined to $b_k$, $a_j$ is joined to $b_\ell$, and $a_i$ is below $a_j$.
The two types of transpositions that reduce the length of the matching corresponding to $\Sigma$ are:
\begin{enumerate}
	\item If $k<j\leq \ell$, then make $\sigma_i=\ell$ and $\sigma_j=k$, i.e. join $a_i$ to $b_\ell$ and $a_j$ to $b_k$, see Figure~\ref{fig:Transpositions}(a).
	This transposition reduces the length of the matching because the point $b_\ell$ is above the right branch of hyperbola $H_{i,j,k}$.
	We call this transposition to be of type I.

	\item If $i<k<\ell$, then make $\sigma_i=\ell$ and $\sigma_j=k$, see Figure~\ref{fig:Transpositions}(b).
	This transposition reduces the length of the matching because of the triangle inequality (the four points are always in convex position). 
	We call this transposition to be of type II.
\end{enumerate}

    \begin{figure}[ht!]
    	\begin{subfigure}[t]{.49\textwidth}
    		\centering
    		\includegraphics[width=\linewidth]{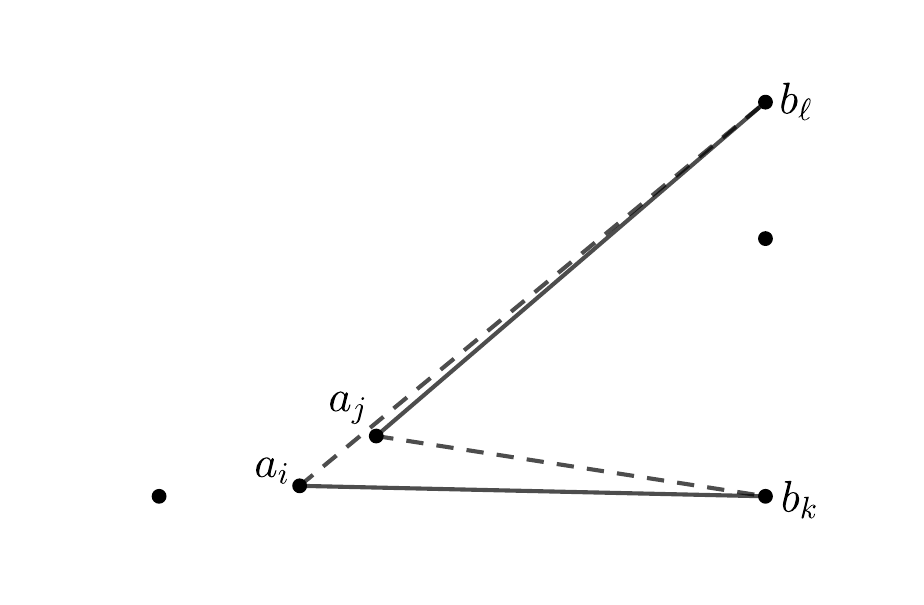}  
    		\caption{A transposition of type I}
    		\label{fig:Transpositions1}
    	\end{subfigure}
    	~
    	\begin{subfigure}[t]{.49\textwidth}
    		\centering
    		\includegraphics[width=\linewidth]{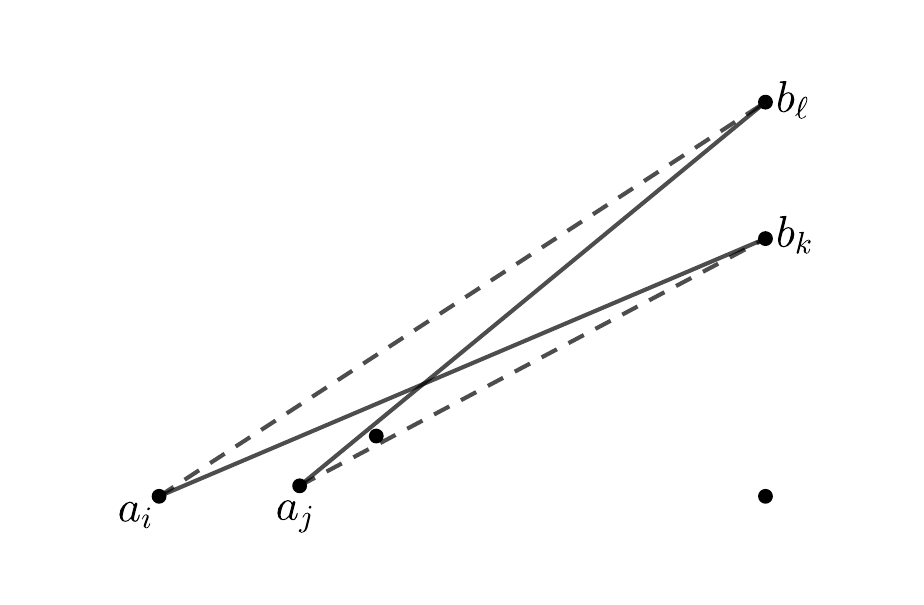}  
    		\caption{A transposition of type II}
    		\label{fig:Transpositions2}
    	\end{subfigure}
    	\caption{Transpositions that reduce the length of the matching: Continuous lines represent the matching before the transposition. Dashed lines represent the matching after the transposition}
    	\label{fig:Transpositions}
    \end{figure}

Let $\Sigma=\{\sigma_1,\ldots,\sigma_m\}$ be any permutation on $[m]$.
We will prove that there is a sequence of permutations $\Sigma_0=Id, \Sigma_1,\ldots,\Sigma_r=\Sigma$ such that $\Sigma_i$ is obtained from $\Sigma_{i-1}$ by applying a transposition that reduces the length of the matching, $1\leq i\leq r$.
The procedure consists in placing the numbers $1,\ldots,m$, one by one, from the smallest to the greatest in the position they have in $\Sigma$ by means of the aforementioned transpositions.
This is done in such a way that, after the numbers $1,2,\ldots,k$ have been placed in their corresponding positions, the numbers $k+1,\ldots,m$ lie in that order in the remaining positions of the current permutation.
The first step is the following: transpose $1$ with $2$, then $1$ with $3$, and so on until $1$ lies in the position it has in $\Sigma$.
Note that all these transpositions are of type I.

Suppose we have placed the numbers $1,\ldots,k$ in their corresponding positions in $\Sigma$ in such a way that the numbers greater than $k$ are sorted in the remaining positions.
Thus, in the next step transpose $k+1$ with $k+2$, then transpose $k+1$ with $k+3$ and so on until $k+1$ lies in the position it has in $\Sigma$.

Let us observe that the aforementioned transpositions reduce the length of the matching.
Suppose that the transposition we make interchanges $k+1$ and $\ell$, where $k+1<\ell\leq m$.
Let $i$ be the current position of $k+1$ and $j$ be the current position of $\ell$. 
Note that, up to the current step, the number $\ell$ has only been moved to the left and, thus, $j\leq \ell$.
If $k+1<j$, then the transposition is of type~I.
If $k+1\geq j$, then, as $j > i$, the transposition is of type~II.
Note that after these transpositions are performed, the elements greater than $k+1$ remain sorted.

In this manner, we can transform the matching $\{s_1,\ldots,s_m\}$ into any other potential matching by means of several steps, in such a way that in each step the matching is shorter than the previous one.
Therefore, we conclude that the matching $\{s_1,\ldots,s_m\}$, which has no crossings, is the longest one in this point set.

\end{proof}

We close this section with the following conjecture:

\begin{conj}
The longest Hamiltonian cycle of any point set always has two edges that cross.
\end{conj}

We also ask the following closely related question: Is it true that the longest Hamiltonian path of a point set must contain two edges that cross each other?

\section {Intersecting families of triangles.}\label{intfam}

In this section we address Lara and Rubio-Montiel's conjecture that any point set $S$ admits a family of intersecting edge disjoint triangles with a quadratic number of elements.
We recall the next result by Buck and Buck~\cite{buck1949equipartition}:

\begin{thm}
	Let $S$ be a set of $n$ points in the plane in general position. Then there are three concurrent lines $\mathcal{L}_1$, $\mathcal{L}_2$ and $\mathcal{L}_3$ that split the plane into six wedges, each containing at least $\lfloor \frac{n}{6}\rfloor$ elements of $S$.
	\label{thm:Buck}
\end{thm}

Suppose that $S$ has $n=6m$ elements. 
By Theorem~\ref{thm:Buck}, there exist three lines, $\mathcal{L}_1$, $\mathcal{L}_2$, and $\mathcal{L}_3$, that intersect at a point $p$ and split the plane into six regions each containing $m$ elements of $S$ in their interior. 
Label the regions generated by $\mathcal{L}_1$, $\mathcal{L}_2$ and $\mathcal{L}_3$ as $\mathcal W_1, \ldots , \mathcal W_6$ as shown in Figure~\ref{fig:Cycle3}.

Label the points in the interior of $\mathcal W_1$ with labels $a_1, \ldots , a_m$ according to their distance to $\mathcal L_2$ such that if $i < j$, then the distance of point $a_i$ to $\mathcal L_2$ is smaller than the distance of $a_j$ to $\mathcal L_2$. 
In a similar way, label the points in $\mathcal W_3$ (respectively in $\mathcal W_5$) as $b_1, \ldots , b_m$ (respectively $c_1, \ldots , c_m$) according to their distance to $\mathcal L_1$ (respectively $\mathcal L_3$).

For $i,j,k$ consider the triangle $T_{i,j,k}$ with vertices $a_i,b_j,c_k$.
We say that $(i,j,k)$ is dominated by $(i',j',k')$ (and denote this as $(i,j,k) \mathcal \leq (i',j',k')$) if 
$i \leq i'$, $j \leq j'$ and $k \leq k'$.
It is easy to see that if $(i,j,k) \mathcal \leq (i',j',k')$
and $(i',j',k') \mathcal \leq (i,j,k)$ do not hold, then
the triangles $T_{i,j,k}$ and $T_{i',j',k'}$ are edge disjoint and
have edges that cross, see Figure~\ref{fig:Cycle3}.
It follows now that if $i+j+k=i'+j'+k'$ 
then some of the edges of
$T_{i,j,k}$ and $T_{i',j',k'}$ cross. 
Let $T_{\left\lceil\frac{3(m+1)}{2}\right\rceil}$ be the set of triangles $T_{i,j,k}$ such that $i+j+k=\left\lceil\frac{3(m+1)}{2}\right\rceil$, $1 \leq i,j,k \leq m$. A simple counting argument shows that $T$ has $\left\lceil\frac{3m^2}{4}\right\rceil$ elements and thus we have:

\begin{thm}\label{mutuallycrossingtriangles}
Any set $S$ with $n=6m$ elements always admits an intersecting family of edge disjoint triangles with $\left\lceil\frac{3m^2}{4}\right\rceil$ elements.
\label{thm:interstriangles}
\end{thm}

    \begin{figure}[ht!]
		\centering
		\includegraphics[width=.6\textwidth]{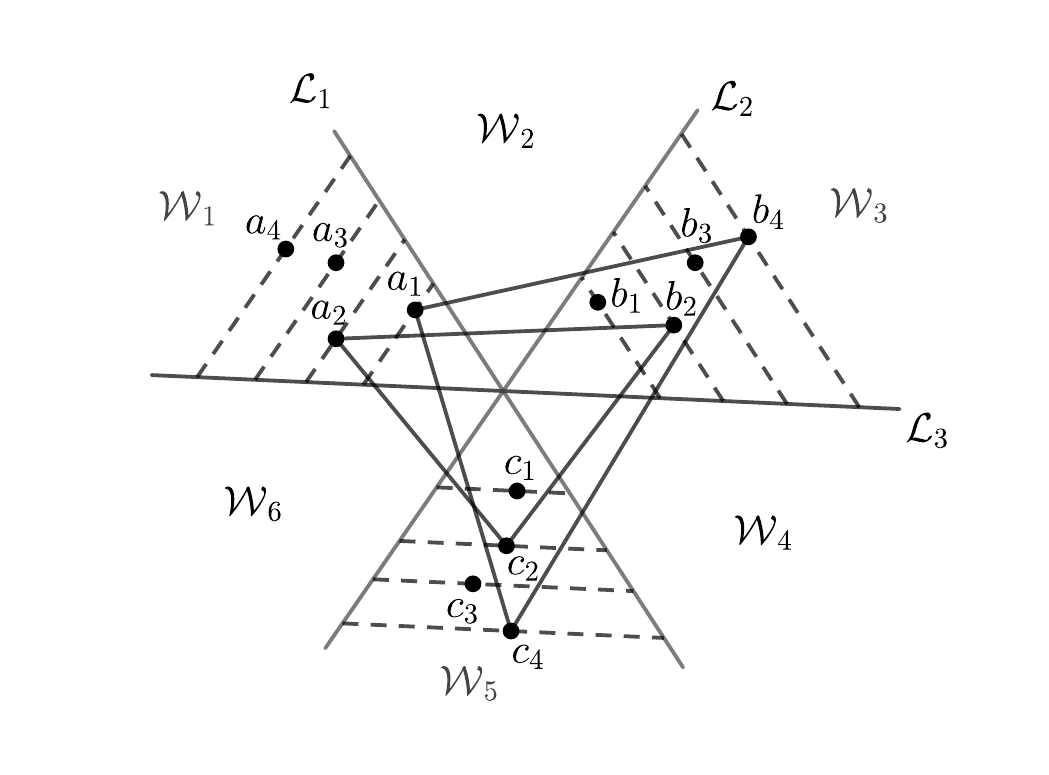}  
		\caption{Triangles $T_{1,4,4}$ and $T_{2,2,2}$ are shown. Since $(1,4,4)\leq(2,2,2)$ and $(2,2,2)\leq (1,4,4)$ do not hold, the two triangles cross}
		\label{fig:Cycle3}
    \end{figure}

We show next that the technique that was employed in the proof of Theorem \ref{mutuallycrossingtriangles} cannot be used to improve the bound of our result any further.

Let $\mathcal{P} =\{a_1, \ldots , a_n, b_1, \ldots , b_n, c_1, \ldots , c_n\}$ be a point set such that the points in each of the sets $\{a_1, \ldots, a_n\}$, $\{b_1, \ldots , b_n\}$ and
$\{c_1, \ldots , c_n\}$ are almost aligned and close to three concurrent semi-lines, see Figure~\ref{fig:Cycle4}. 
Then $T_{i,j,k}$ and $T_{i',j',k'}$ cross if and only if $(i,j,k)$ does not dominate $(i',j',k')$ and viceversa.
Next, we prove the following result:

    \begin{figure}[ht!]
		\centering
		\includegraphics[width=.6\textwidth]{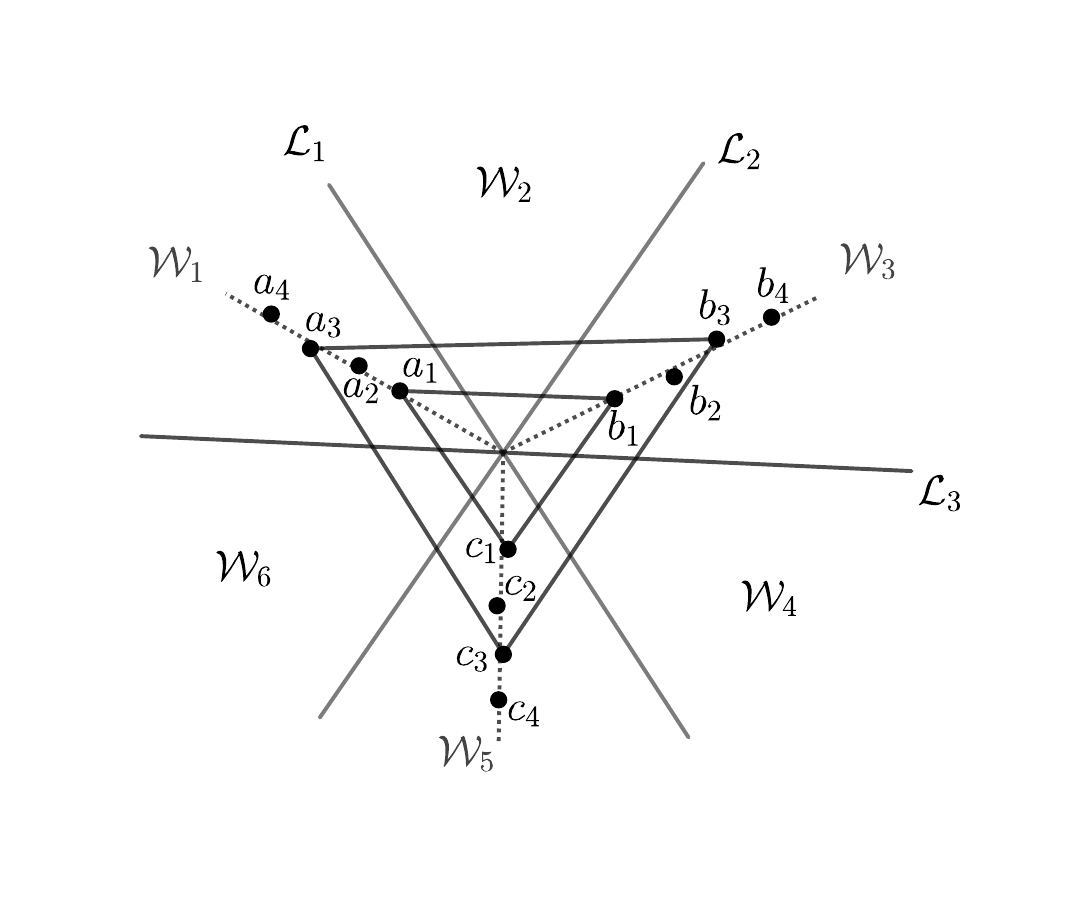}  
		\caption{}
		\label{fig:Cycle4}
    \end{figure}

\begin{lem}
\label{lem:13}
Let $\mathcal{F}$ be a family of distinct ordered triplets $(i,j,k)$ of integers between $1$ and $n$. 
Suppose that there are no two triplets $(i,j,k)$, $(i',j',k')$ in $\mathcal{F}$ with $i\leq i', j\leq j', k\leq k'$, then $|F|\leq\left\lceil\frac{3n^2}{4}\right\rceil$.
\end{lem}

\begin{proof}
By taking the first two coordinates of every triplet in $\mathcal{F}$ we obtain a collection $\mathcal{F}'$ of distinct lattice points in the plane whose coordinates lie between $1$ and $n$.
Partition the set $\mathcal Q=\{(i,j)\in\mathbb{Z}^2:1\leq i,j\leq n\}$ into $n$ chains as shown in Figure~\ref{fig:Chains}. 
Observe that these chains have respectively $2n-1, 2(n-1)-1, \ldots , 1$ elements. 

Observe that if two elements $(i,j)$, $(i',j')$ of any of these chains are such that $(i,j,k)$ and $(i,j',k')$ belong to $\mathcal F$, then $k \neq k'$, otherwise one of these triplets would majorize the other.
It follows now that any of these chains contains at most $\max \{2i-1, n\}$ elements $(i,j)$ such that $(i,j,k)$ belong to $\mathcal F$, $i=1,\ldots,n$. Summing over all chains, we get that \[|\mathcal{F}|=|\mathcal{F}'|\leq\sum_{i=1}^n \max\{2i-1,n\}=\left(\sum_{i=1}^{\left\lceil\frac{n}{2}\right\rceil}2i-1\right)+\left\lfloor\frac{n}{2}\right\rfloor n=\left\lceil\frac{3n^2}{4}\right\rceil.\]
\end{proof}

    \begin{figure}[ht!]
		\centering
		\includegraphics[width=.37\textwidth]{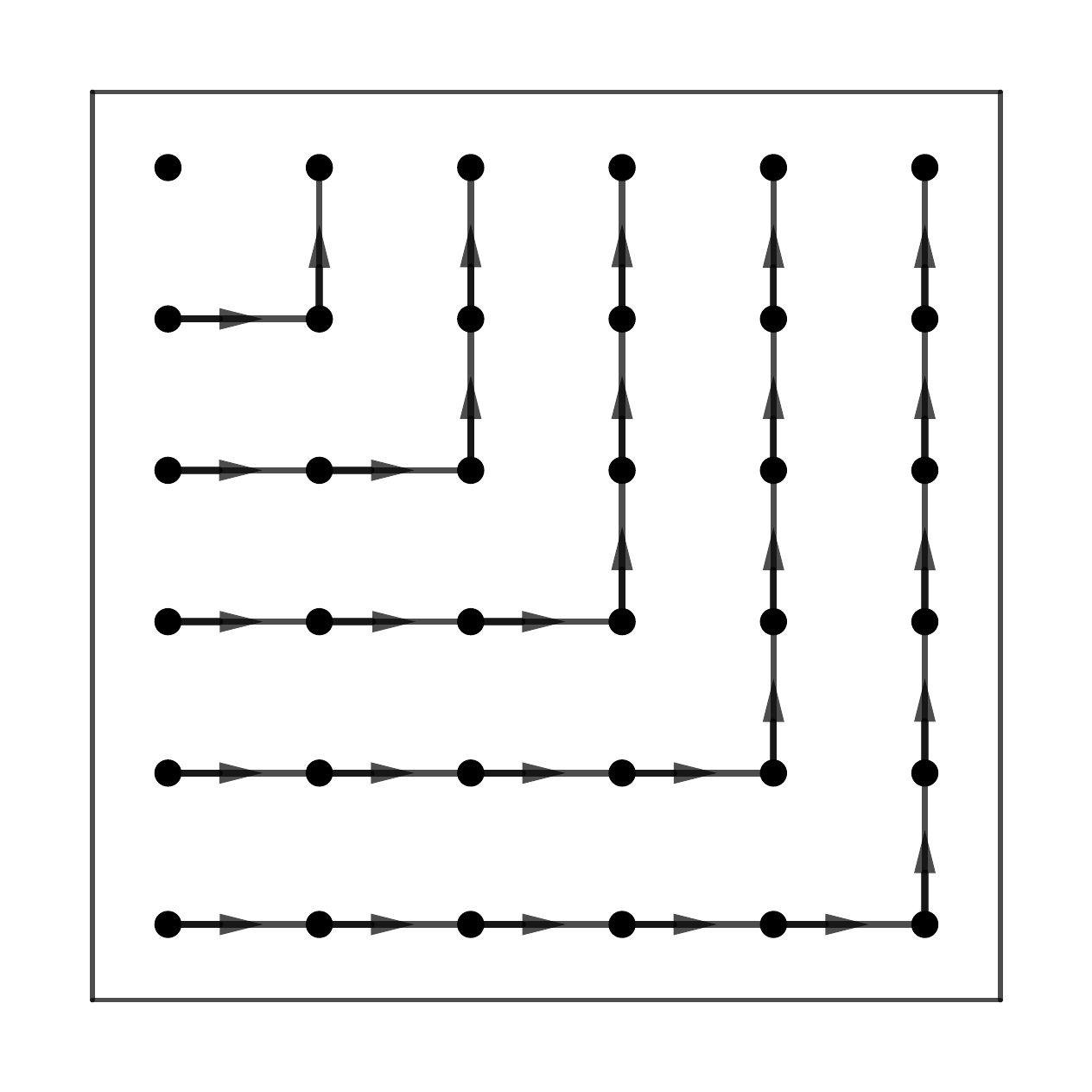}  
		\caption{Illustration of the proof of Lemma~\ref{lem:13}}
		\label{fig:Chains}
    \end{figure}

The next result follows:

\begin{thm} \label{mutcross}
$\mathcal P$ contains exactly $\left\lceil\frac{3n^2}{4}\right\rceil$ mutually crossing edge disjoint triangles, such that each triangle has a vertex in each of $\{a_1, \ldots , a_n\}$, $\{b_1, \ldots , b_n\}$, and $\{ c_1, \ldots , c_n\}$.
\end{thm}

We believe that the bound of Theorem~\ref{mutcross} is optimal, and that any family of
mutually crossing edge disjoint triangles with vertices in $\mathcal P$ has at most
$\left\lceil\frac{3n^2}{4}\right\rceil$ elements.

\subsection{Points in convex position}

We prove now that any set with $3n$ points in convex position has an intersecting family of edge disjoint triangles with at least $n^2$ triangles.

Let $\mathcal P=\{a_1, \ldots , a_n, b_1, \ldots , b_n, c_1, \ldots , c_n\}$ be a set with $3n$ points, and assume without loss of generality that the elements of $\mathcal P$ lie on a circle as shown in Figure~\ref{fig:TriangCirc} (a). 
We observe first that any two edge disjoint triangles
with vertices $a_i,b_j,c_k$ and $a_{i'}, b_{j'},c_{k'}$, respectively, have edges that cross. 
Thus if we can partition all of the edges of the geometric complete tripartite graph $G_{\mathcal A,\mathcal B,\mathcal C}$ into $n^2$ edge disjoint triangles, we are done, where $\mathcal A=\{a_1, \ldots , a_n\}$, $\mathcal B=\{b_1, \ldots , b_n\}$, and $\mathcal C=\{c_1, \ldots , c_n\}$.

It is well known that the edges of the complete bipartite graph $K_{n,n}$ can be partitioned into $n$ edge disjoint perfect matchings; we apply this result to $K_{\mathcal B,\mathcal C}$ to obtain a partition of its edges into $n$ matchings $\mathcal M_1, \ldots , \mathcal M_n$.
Each edge joining a point in $\mathcal B$ with a point in $\mathcal C$ together with a point in $\mathcal A$ defines a triangle. 
For each $1 \leq i \leq n$ consider the set of triangles $T_i$ defined by all the edges in $\mathcal M_i$ and the point $a_i\in \mathcal A$, see Figure~\ref{fig:TriangCirc} (b).
Thus $T_1 \cup \ldots \cup T_n$ is an intersecting family with $n^2$ elements. 

    \begin{figure}[ht!]
    	\begin{subfigure}[t]{.45\textwidth}
    		\centering
    		\includegraphics[width=\linewidth]{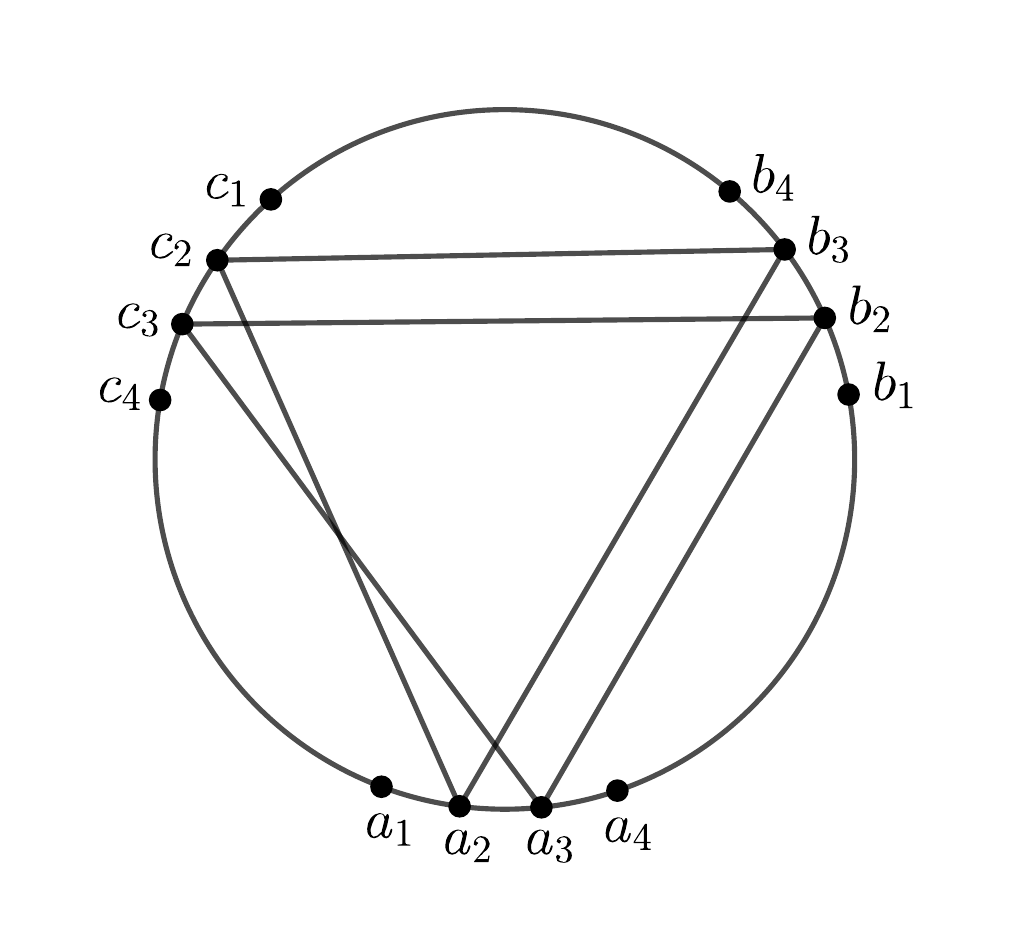}  
    		\caption{}
    		\label{fig:TriangCirc_1}
    	\end{subfigure}
    	~
    	\begin{subfigure}[t]{.45\textwidth}
    		\centering
    		\includegraphics[width=\linewidth]{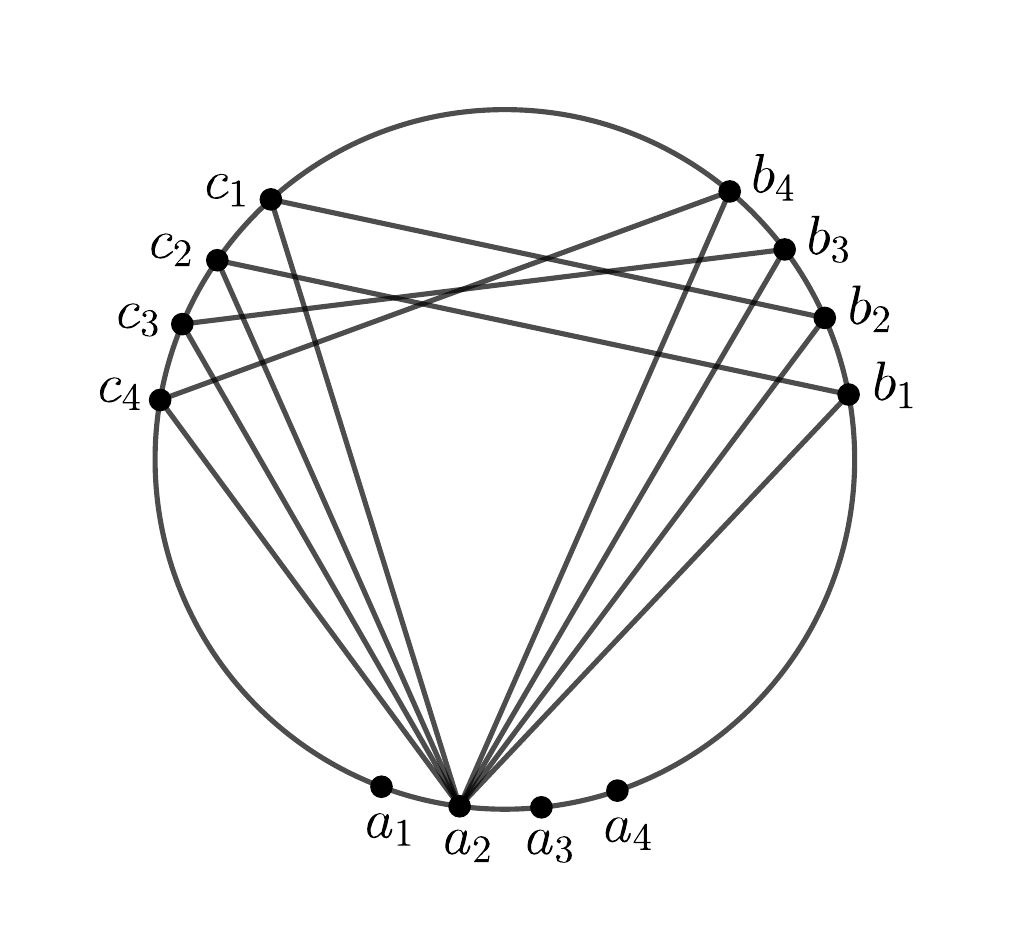}  
    		\caption{}
    		\label{fig:TriangCirc_2}
    	\end{subfigure}
    	\caption{Illustration of the proof of Theorem~\ref{thm:convex_pos}}
    	\label{fig:TriangCirc}
    \end{figure}

\begin{thm}
\label{thm:convex_pos}
Any set with $3n$ points in convex position has an intersecting family with $n^2$ edge disjoint triangles.
\end{thm}

We finish our paper with the following conjecture:
 
\begin{conj}
Any point set with $3n$ elements has at most $n^2$ intersecting edge disjoint triangles.
\end{conj}

\bibliographystyle{plainurl}
\bibliography{refs}
\end{document}